\documentclass[11pt]{amsart}

\newtheorem{thm}{Theorem}[section]
\newtheorem*{thm*}{Theorem}
\newtheorem{lem}[thm]{Lemma}
\newtheorem*{lem*}{Lemma}

\newtheorem{claim}[thm]{Claim}
\newtheorem{prop}[thm]{Proposition}

\theoremstyle{definition}

\newtheorem*{case*}{Case}

\newtheorem{defn}[thm]{Definition}
\newtheorem*{defn*}{Definition}
\newtheorem{exmp}[thm]{Example}
\newtheorem*{exmp*}{Example}

\newtheorem{step}{Step}\renewcommand{\thestep}{}

\theoremstyle{remark}

\renewcommand{\thecase}{}

\newtheorem{rmk}[thm]{Remark}
\newtheorem*{rmk*}{Remark}

\makeatletter
\def\alphenumi{
  \def\theenumi{\alph{enumi}}
  \def\p@enumi{\theenumi}
  \def\labelenumi{(\@alph\c@enumi)}}
\makeatother

\makeatletter
\def\thecase{\@arabic\c@case}
\makeatother

\makeatletter
\def\thestep{\@arabic\c@step}
\makeatother

\newcount\hh
\newcount\mm
\mm=\time
\hh=\time
\divide\hh by 60
\divide\mm by 60
\multiply\mm by 60
\mm=-\mm
\advance\mm by \time
\def\hhmm{\number\hh:\ifnum\mm<10{}0\fi\number\mm}

\let\oldmarginpar\marginpar
\renewcommand\marginpar[1]{\-\oldmarginpar[\raggedleft\footnotesize #1]%
{\raggedright\footnotesize #1}}

\renewcommand\emptyset{\varnothing}

\newcommand\HH{\mathbb{H}}

\newcommand\NN{\mathbb{N}}

\newcommand\RR{\mathbb{R}}

\newcommand\fb{{\mathfrak{b}}}

\newcommand\fK{{\mathfrak{K}}}

\newcommand\sC{{\mathscr{C}}}

\newcommand\sO{{\mathscr{O}}}

\newcommand\sQ{{\mathscr{Q}}}

\newcommand\sS{{\mathscr{S}}}

\newcommand\sU{{\mathscr{U}}}

\newcommand\eps{\varepsilon}

\newcommand\less{\setminus}

\DeclareMathOperator{\mydirac}{\slashed{\partial}}

\newcommand{\essinf}{\operatornamewithlimits{ess\ inf}}
\newcommand{\esssup}{\operatornamewithlimits{ess\ sup}}

\DeclareMathOperator{\Int}{int}

\newcommand\loc{\operatorname{loc}}

\newcommand\supp{\operatorname{supp}}

\newcommand\tr{\operatorname{tr}}

\numberwithin{equation}{section}

\usepackage{amssymb}
\usepackage{graphicx}
\usepackage{hyperref}
\usepackage{mathrsfs}
\usepackage[usenames]{color}
\usepackage{slashed}
\usepackage{url}
\usepackage{verbatim}

\hypersetup{pdftitle={Maximum principles for boundary-degenerate linear parabolic differential operators}}
\hypersetup{pdfauthor={Paul M. N. Feehan}}

\begin{document}

\title[Maximum principles for boundary-degenerate parabolic operators]{Maximum principles for boundary-degenerate linear parabolic differential operators}
\author[Paul M. N. Feehan]{Paul M. N. Feehan}
\address{Department of Mathematics, Rutgers, The State University of New Jersey, 110 Frelinghuysen Road, Piscataway, NJ 08854-8019}
\email{feehan@math.rutgers.edu}

\date{July 19, 2013}

\begin{abstract}
We develop weak and strong maximum principles for boundary-degenerate, linear, parabolic, second-order partial differential operators, $Lu := -u_t-\tr(aD^2u)-\langle b, Du\rangle + cu$, with \emph{partial} Dirichlet boundary conditions. The coefficient, $a(t,x)$, is assumed to vanish along a non-empty open subset, $\mydirac_0\!\sQ$, called the \emph{degenerate boundary portion}, of the parabolic boundary, $\mydirac\!\sQ$, of the domain $\sQ\subset\RR^{d+1}$, while $a(t,x)$ may be non-zero at points in the \emph{non-degenerate boundary portion}, $\mydirac_1\!\sQ := \mydirac\!\sQ\less\overline{\mydirac_0\!\sQ}$. Points in $\mydirac_0\!\sQ$ play the same role as those in the interior of the domain, $\sQ$, and only the non-degenerate boundary portion, $\mydirac_1\!\sQ$, is required for boundary comparisons. We also develop comparison principles and a priori maximum principle estimates for solutions to boundary value and obstacle problems defined by boundary-degenerate parabolic operators, again where only the non-degenerate boundary portion, $\mydirac_1\!\sQ$, is required for boundary comparisons. Our results complement those in our previous articles \cite{Feehan_maximumprinciple_v1, Feehan_perturbationlocalmaxima}.
\end{abstract}

%
%
%
%



\thanks{The author was partially supported by NSF grant DMS-1237722, a visiting faculty appointment in the Department of Mathematics at Columbia University, and the Max Planck Institut f\"ur Mathematik, Bonn.}

\maketitle
\tableofcontents
\listoffigures

\section{Introduction}
\label{sec:Introduction}
The weak maximum principle for a parabolic, possibly degenerate, linear, second-order partial differential operator in non-divergence form, $Lu=-u_t-\tr(aD^2u)-\langle b, Du\rangle + cu$, provides uniqueness of solutions, $u$, to boundary value problems on an open subset $\sQ\subset\RR^{d+1}$ with Dirichlet condition prescribed on the \emph{full} parabolic boundary, $\mydirac\!\sQ$, when the solutions belong to $C^2(\sQ)$ or $W^{2,d+1}_{\loc}(\sQ)$ \cite{Krylov_LecturesHolder, Lieberman}, or $C(\sQ)$ if interpreted in the viscosity sense \cite{Crandall_Ishii_Lions_1992}. As noted by G. Fichera \cite{Fichera_1956, Fichera_1960}  (see also the expositions due to O. A. Ole{\u\i}nik and E. V. Radkevi{\v{c}} \cite{Oleinik_Radkevic, Radkevich_2009a, Radkevich_2009b}), one can obtain uniqueness of solutions to boundary value problems with Dirichlet condition prescribed only along a \emph{part} of the parabolic boundary, $\mydirac_1\!\sQ:=\mydirac\!\sQ\less\overline{\mydirac_0\!\sQ}$, for a non-empty, open subset $\mydirac_0\!\sQ\subseteqq\mydirac\!\sQ$, when the coefficient $a(t,x)$ vanishes along $\mydirac_0\!\sQ$ (we call such an operator, $L$, \emph{boundary-degenerate}) and the \emph{Fichera function}\footnote{Namely, $\fb := (b^k-a^{kj}_{x_j})n_k$, where $(n_0,n_1,\ldots,n_d)$ is the inward-pointing unit normal vector field along $\mydirac_0\!\sQ$ \cite[Equation (1.1.3)]{Radkevich_2009a}.},
$\fb$, defined by $L$ and $\mydirac_0\!\sQ$
obeys the Fichera sign condition \cite[p. 308]{Radkevich_2009a} along $\mydirac_0\!\sQ$.

When the operator, $L$, is given in divergence form, so one can define a weak solution, $u \in W^{1,2}(\sQ)$, to a boundary value problem, one can also obtain uniqueness of solutions with partial Dirichlet data when the Fichera sign condition holds along $\mydirac_0\!\sQ$ \cite{Fichera_1956, Fichera_1960, Oleinik_Radkevic, Radkevich_2009a, Radkevich_2009b}.

\begin{figure}
 \centering
 \begin{picture}(210,210)(0,0)
 \put(0,0){\includegraphics[scale=0.6]{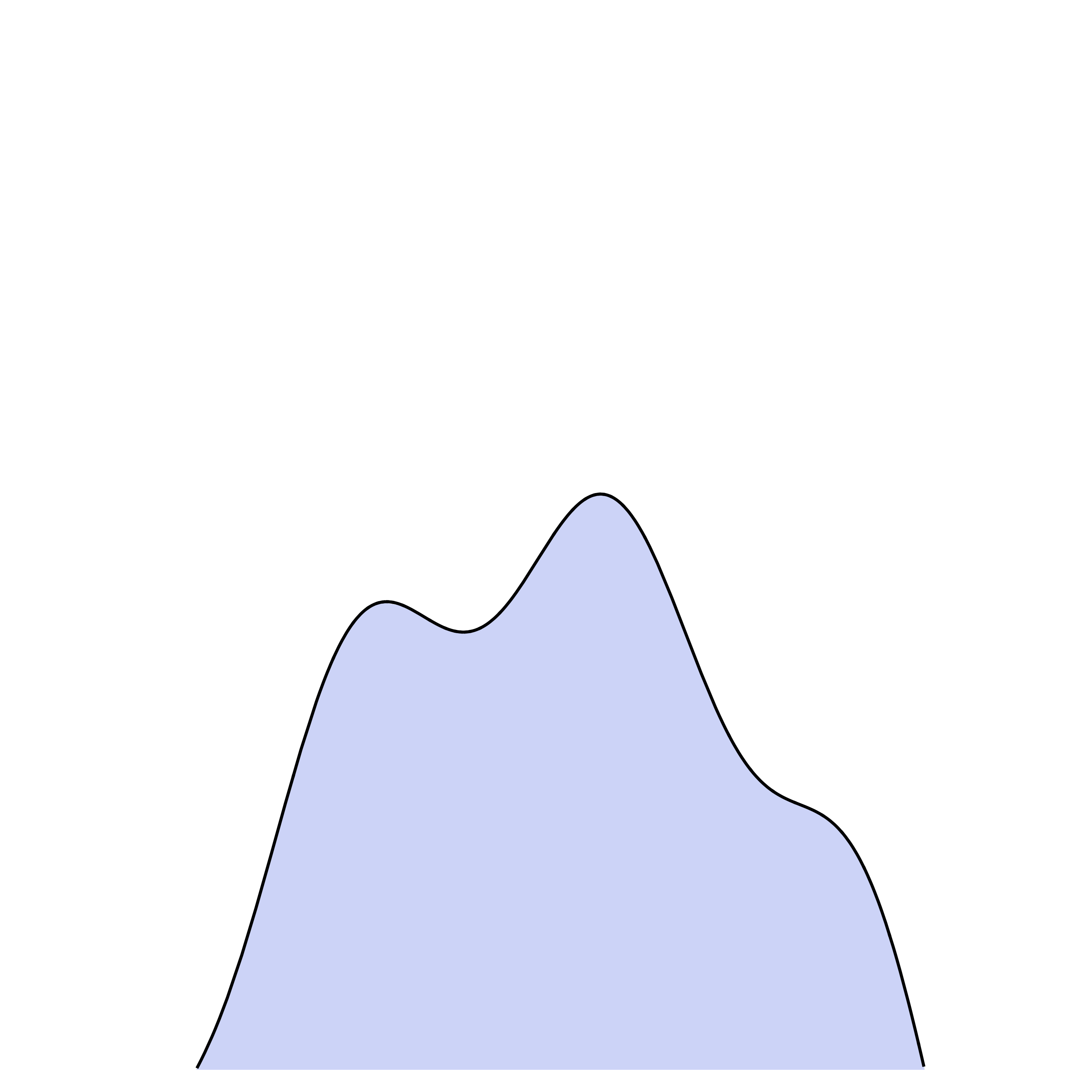}}
 \put(90,10){$\mydirac_0\!\sQ$}
 \put(105,50){$\sQ$}
 \put(110,95){$\mydirac_1\!\sQ$}
 \put(150,80){$\RR^{d+1}$}
 \end{picture}
 \caption[A domain and its `degenerate' and `non-degenerate' boundaries.]{A subdomain, $\sQ\subset\RR^{d+1}$, and its `degenerate' and `non-degenerate' boundaries, $\mydirac_0\!\sQ$ and $\mydirac_1\!\sQ$. In maximum principles, the degenerate boundary portion, $\mydirac_0\!\sQ$, plays the same role as the interior of the domain, $\sQ$.}
 \label{fig:domain}
\end{figure}

However, the Fichera weak maximum principle does not take into account a more modern view of the appropriate function spaces in which uniqueness is sought, such as those used by P. Daskalopoulos and the author \cite{Daskalopoulos_Feehan_statvarineqheston}, Daskalopoulos, R. Hamilton, and E. Rhee \cite{DaskalHamilton1998, Daskalopoulos_Rhee_2003}, E. Ekstr\"om and J. Tysk \cite{Ekstrom_Tysk_bcsftse}, C. L. Epstein and R. Mazzeo \cite{Epstein_Mazzeo_annmathstudies}, C. A. Pop and the author \cite{Feehan_Pop_mimickingdegen_pde}, and H. Koch \cite{Koch}. Indeed, the Fichera weak maximum principles lead to the imposition of additional Dirichlet boundary conditions which are not necessarily motivated by the underlying application, whether in biology, finance, or physics. These additional Dirichlet boundary conditions, usually for certain ranges of parameters defining the operator, $L$, are often less natural than the physically-motivated regularity properties suggested by choices of appropriate weighted H\"older spaces \cite{DaskalHamilton1998, Daskalopoulos_Rhee_2003, Epstein_Mazzeo_annmathstudies, Feehan_Pop_mimickingdegen_pde} or Sobolev spaces \cite{Daskalopoulos_Feehan_statvarineqheston, Feehan_Pop_regularityweaksoln, Koch}, which automatically encode enough regularity up to the portion, $\mydirac_0\!\sQ$, of the parabolic boundary where the operator, $L$, becomes degenerate.

However, the question of exactly how regular the solution should be near $\mydirac_0\!\sQ$ is delicate. If we ask for too much regularity, such as $C^2$ up to $\mydirac_0\!\sQ$, we may obtain uniqueness but have no existence theory. Indeed, denoting $\underline\sQ := \sQ\cup \mydirac_0\!\sQ$, this was the motivation for the introduction by Daskalopoulos and Hamilton of their weighted H\"older space, $C^{2+\alpha}_s(\underline\sQ)$, in \cite[pp. 901--902]{DaskalHamilton1998} for the purpose of solving the Cauchy problem for a boundary-degenerate, linear, second-order, parabolic operator, $L$, arising in the study of the porous medium equation. The weighted H\"older space, $C^{2+\alpha}_s(\underline\HH_T)$, plays a key role in the proofs due to Daskalopoulos, Hamilton, and Rhee of both existence and uniqueness of solutions to the Cauchy problem on $\HH_T$, where $\HH := \RR^{d-1}\times\RR_+$. On the other hand, if we ask for too little regularity, such as $C^0$ up to $\mydirac_0\!\sQ$, the examples of the Heston stochastic volatility model \cite{Daskalopoulos_Feehan_statvarineqheston, Daskalopoulos_Feehan_optimalregstatheston} in mathematical finance, the porous medium equation \cite{DaskalHamilton1998}, Wright-Fisher diffusion model in mathematical biology \cite{Epstein_Mazzeo_annmathstudies}, and interest rate models in mathematical finance \cite{Ekstrom_Tysk_bcsftse} indicate that this usually leads to the imposition of an unphysical Dirichlet boundary condition. In particular, this has the unintended consequence that the unique solutions selected by the Fichera weak maximum principle can be no more than continuous up to the boundary; a detailed example illustrating this point with the aid of the Kummer equation is provided by the author in \cite[\S 1.1]{Feehan_perturbationlocalmaxima}.

Let $\sS(d)\subset\RR^{d\times d}$ denote the subset of symmetric matrices and $\sS^+(d)\subset\sS(d)$ denote the subset of non-negative definite, symmetric matrices. In the context of maximum principles for boundary-degenerate parabolic operators, a useful concept of boundary regularity is given by the

\begin{defn}[Second-order boundary condition and boundary regularity]
\label{defn:Second-order_boundary_regularity}
Let $\sQ\subset\RR^{d+1}$ be an open subset and $a:\sQ\to \sS^+(d)$ be a function. We say that $u\in C^2(\sQ)\cap C^1(\underline \sQ)$ obeys a \emph{second-order boundary condition} along $\mydirac_0\!\sQ$ if
\begin{equation}
\label{eq:Ventcel}
\tr(a D^2u) \in C(\underline \sQ) \quad\hbox{and}\quad \tr(a D^2u) = 0 \quad\hbox{on }\mydirac_0\!\sQ,
\end{equation}
and write $u \in C^2_s(\underline \sQ)$ if $u\in C^2(\sQ)\cap C^1(\underline \sQ)$ obeys \eqref{eq:Ventcel}.
\end{defn}

Given an open subset $\sQ\subset\RR^{d+1}$, we shall say that a function $u \in C^2(\sQ)$ (respectively, $W^{2,d+1}_{\loc}(\sQ)$) is \emph{(strictly) $L$-subharmonic} if $Lu\leq 0$ (respectively, $Lu < 0$) (a.e.) on $\sQ$. (The notation will be explained below.)

The purpose of this article is to develop weak and strong maximum principles for $L$-subharmonic functions in $C^2_s(\underline\sQ)$, when $L$ is a boundary-degenerate parabolic operator in non-divergence form, and Dirichlet boundary conditions are prescribed only along $\mydirac_1\!\sQ$. Our results complement those in \cite{Feehan_maximumprinciple_v1} for the case of boundary-degenerate elliptic operators, $A$, and $A$-subharmonic functions in $C^2_s(\underline\sO)$. We develop a priori maximum principle estimates for solutions, subsolutions, and supersolutions in $C^2(\sQ)$ or $W^{2,d+1}_{\loc}(\sQ)$ to boundary value problems for $L$, with Dirichlet boundary conditions prescribed only along $\mydirac_1\!\sQ$. We also develop comparison principles and a priori maximum principle estimates for solutions and supersolutions in $W^{2,d+1}_{\loc}(\sQ)$ to unilateral obstacle problems for $L$, again with Dirichlet boundary conditions prescribed only along $\mydirac_1\!\sQ$.

While the focus of this article is on the development of weak and strong maximum principles for subsolutions in $C^2(\sQ)$ or $W^{2,d+1}_{\loc}(\sQ)$ to \emph{linear} boundary-degenerate parabolic equations in non-divergence form, it appears likely that our approach can be extended to give comparison principles for viscosity subsolutions and supersolutions to \emph{fully nonlinear} equations on $\sQ$ with fully nonlinear boundary conditions imposed only on $\mydirac_1\!\sQ$, provided the concept of viscosity solution \cite{Crandall_Ishii_Lions_1992} is appropriately modified. Similarly, one can expect analogues of the classical Aleksandrov-Bakelman-Pucci estimates (compare \cite[Theorem 7.1]{Lieberman} for functions in $W^{2,d+1}_{\loc}(\sQ)$), where the role of the full parabolic boundary, $\mydirac\!\sQ$, would be replaced by the non-degenerate boundary portion, $\mydirac_1\!\sQ$. These ideas will be developed in a separate article.

Our companion article \cite{Feehan_perturbationlocalmaxima} develops weak and strong maximum principles for $L$-subharmonic functions in both $C^2(\sQ)\cap C^1(\underline\sQ)$ and $W^{2,d+1}_{\loc}(\sQ)\cap C^1(\underline\sQ)$ for a boundary-degenerate parabolic operator, $L$, in non-divergence form, given slightly stronger boundary-regularity conditions on the coefficients $a$ and $b$. However, the methods in \cite{Feehan_perturbationlocalmaxima} are quite different to those used in the present article.

In \cite{Feehan_classical_perron_parabolic}, we apply the main results of this article in our proof of existence of solutions to the parabolic equation \eqref{eq:Parabolic_equation} and obstacle problem \eqref{eq:Parabolic_obstacle_problem} with Dirichlet boundary condition \eqref{eq:Parabolic_boundary_condition} using a version of the classical Perron method \cite[\S 3.4]{Lieberman}.

A weak maximum principle for the parabolic (model) Kimura diffusion operator is given by Epstein and Mazzeo in \cite[Proposition 4.1.1]{Epstein_Mazzeo_annmathstudies}, who also employ a form of second-order boundary condition, together with a Hopf lemma and a strong maximum principle in \cite[Lemma 4.2.4 and 4.2.5]{Epstein_Mazzeo_annmathstudies}. Related uniqueness results and weak maximum principles for classical (sub-)solutions to second-order, linear, degenerate elliptic and parabolic operators are proved by M. A. Pozio, F. Punzo,
and A. Tesei in \cite{Pozio_Punzo_Tesei_2008, Punzo_Tesei_2009a, Punzo_Tesei_2009b}, but they do not make use of a second-order boundary regularity condition, such as $C^2_s(\underline\sQ)$, to obtain uniqueness results.

\subsection{Boundary value and obstacle problems for boundary-degenerate, linear, second-order, parabolic partial differential operators}
\label{subsec:Parabolic_operator_boundary_problem}
Consider a possibly non-cylindrical open subset $\sQ\subset\RR^{d+1}$ with topological boundary $\partial \sQ$, where $d\geq 1$. Given $P^0=(t^0,x^0)\in\RR^{d+1}$ and $R>0$, define
\begin{equation}
\label{eq:Parabolic_cylinder}
Q_R(P^0) := \left\{(t,x)\in \RR^{d+1}: \max\left\{|x-x^0|, \, |t-t^0|^{1/2}\right\} < R, \ t>t^0\right\},
\end{equation}
where our time convention is opposite to that of G. Lieberman \cite[p. 5]{Lieberman} since we consider terminal rather than initial boundary problems in this article. Following Lieberman \cite[p. 7]{Lieberman}, for possibly non-cylindrical open subsets, we make the

\begin{defn}[Parabolic boundary]
\label{defn:Parabolic_boundary}
For an open subset $\sQ\subset\RR^{d+1}$, we call
\begin{equation}
\label{eq:Parabolic boundary}
\mydirac\!\sQ := \{P^0=(t^0,x^0)\in\partial \sQ: Q_\eps(P^0)\cap \sQ \neq \emptyset, \ \forall\,\eps > 0\}
\end{equation}
the \emph{parabolic boundary} of $\sQ$.
\end{defn}

We identify the vector spaces $\RR$ and $\RR^d$ with the hyperplanes $\RR\times\{0\}$ and $\{0\}\times\RR^d\subset\RR^{d+1}$ of temporal and spatial vectors, respectively. When the boundary of $\sQ$ has a tangent plane at a point $P \in \partial \sQ$, we write the \emph{inward}-pointing  normal vector as $n_0(P)e_0 + \vec n(P)\in\RR^{d+1}$, where $\vec n(P) = \sum_{i=1}^dn^i(P)e_i$ and $e_0,e_1,\ldots,e_d$ is the standard basis of $\RR^{d+1}$. Given a map
\begin{equation}
\label{eq:a_nonnegative}
a:\sQ \to \sS^+(d),
\end{equation}
we call\footnote{The corresponding definition in the elliptic case, where the $n_0$ vanishing condition is omitted, is a slight generalization of that used by Fichera \cite{Fichera_1960}, Ole{\u\i}nik, and Radkevi{\v{c}} \cite{Oleinik_Radkevic}, \cite[p. 308]{Radkevich_2009a}.}
\begin{equation}
\label{eq:Degeneracy_locus_parabolic}
\mydirac_0\!\sQ := \Int\left\{P\in\partial \sQ: \lim_{\sO\ni P'\to P}a(P') = 0\right\}\cap \Int\{P\in\partial \sQ: n_0(P)=0\},
\end{equation}
the \emph{degenerate parabolic boundary} (again slightly abusing terminology) defined by $a:\sQ \to \sS^+(d)$, where $\Int S$ denotes the interior of a subset $S$ of a topological space. Throughout this article we shall allow $\mydirac_0\!\sQ$ to be non-empty and denote
\begin{equation}
\label{eq:Domain_plus_degenerate_boundary_parabolic}
\underline \sQ := \sQ\cup\mydirac_0\!\sQ.
\end{equation}
We also call
\begin{equation}
\label{eq:Parabolic_nondegeneracy_locus}
\mydirac_1\!\sQ := \Int\left\{P\in\partial \sQ: \lim_{\sO\ni P'\to P}a(P') \neq 0\right\}\cup \Int\{P\in\partial \sQ: n_0(P)\neq 0\}
\end{equation}
the \emph{non-degenerate parabolic boundary} defined by $a:\sQ \to \sS^+(d)$ and observe that
\begin{equation}
\label{eq:Degenerate_boundary_decomposition}
\mydirac\!\sQ = \left(\mydirac\!\sQ\cap\overline{\mydirac_0\!\sQ}\right)\cup\mydirac_1\!\sQ = \mydirac_0\!\sQ\cup\left(\mydirac\!\sQ\cap\overline{\mydirac_1\!\sQ}\right),
\end{equation}
where $\overline{\Sigma}$ indicates closure of a subset $\Sigma\subset\partial\sQ$ with respect to the topological boundary, $\partial\sQ$. The meaning of the different boundary portions is clarified in the following

\begin{exmp}[Boundaries for parabolic cylinders]
\label{exmp:Boundary_cylinder_parabolic}
Given a parabolic cylinder, $\sQ=(0,T)\times\sO=\sO_T$, for some $T>0$ and open subset $\sO\subset\RR^d$, then
\begin{align*}
\mydirac\!\sQ &= \left(\{T\}\times\partial\sO\right) \cup \left(\{T\}\times\sO\right) \cup \left((0,T)\times\partial\sO\right)
\\
&= \left(\{T\}\times\bar\sO\right) \cup \left((0,T)\times\partial\sO\right)
\\
&= \left(\{T\}\times\sO\right) \cup \left((0,T]\times\partial\sO\right),
\end{align*}
where (in the terminology of \cite[p. 7]{Lieberman}), the subset $\{T\}\times\sO$ is the \emph{top} of $\sQ$, and $(0,T)\times\partial\sO$ is the \emph{side} of $\sQ$, and $\{T\}\times\partial\sO$ is the \emph{corner} of $\sQ$.

Unlike its elliptic counterpart, we note that the non-degenerate boundary portion,
\begin{equation}
\label{eq:Nondegenerate_boundaryportion_nonempty_parabolic}
\mydirac_1\!\sQ \hbox{ is always non-empty}.
\end{equation}
For example, when $\sQ=(0,T)\times\sO$, then
$$
\mydirac_1\!\sQ \supset \{T\}\times\sO,
$$
since $n_0(P)=-1$ when $P\in \{T\}\times\sO$, and again keeping in mind our convention of considering terminal, rather than initial boundary problems, in this article because of their association with optimal stopping problems in probability theory.

Let us now suppose that $a(t,x)$ is independent of $t\in\RR$ and write $a(t,x)=a(x)$, for all $(t,x)\in\sQ$. We recall from \cite{Feehan_maximumprinciple_v1} that
$$
\partial_0\sO := \Int\left\{x\in\partial \sO: \lim_{\sO\ni x'\to x}a(x') = 0\right\}
\quad\hbox{and}\quad
\partial_1\sO := \Int\left\{x\in\partial \sO: \lim_{\sO\ni x'\to x}a(x') \neq 0\right\},
$$
and thus
$$
\partial\sO = \partial_0\sO\cup\overline{\partial_1\sO} = \overline{\partial_0\sO}\cup \partial_1\sO.
$$
Furthermore,
$$
\mydirac_0\!\sQ = (0,T)\times\partial_0\sO
\quad\hbox{and}\quad
\underline \sQ = (0,T)\times (\sO\cup \partial_0\sO) = (0,T)\times\underline\sO = \underline\sO_T,
$$
noting that $n_0(P)=0$ for all $P\in \mydirac_0\!\sQ$, while
\begin{align*}
\mydirac_1\!\sQ &= \left(\{T\}\times(\sO\cup\partial_1\sO)\right) \cup \left((0,T)\times\partial_1\sO\right)
\\
&= \left(\{T\}\times\sO\right) \cup \left((0,T]\times\partial_1\sO\right).
\end{align*}
Clearly,
\begin{align*}
\mydirac_0\!\sQ\cup\left(\mydirac\!\sQ\cap\overline{\mydirac_1\!\sQ}\right)
&= \left((0,T)\times\partial_0\sO\right)\cup \left(\{T\}\times\bar\sO\right) \cup \left((0,T]\times\overline{\partial_1\sO}\right)
\\
&= \left(\{T\}\times\sO\right) \cup \left((0,T]\times\partial\sO\right) = \mydirac\!\sQ
\\
&= \left((0,T]\times\overline{\partial_0\sO}\right)\cup \left(\{T\}\times\sO\right) \cup \left((0,T]\times\partial_1\sO\right)
\\
&= \left(\mydirac\!\sQ\cap\overline{\mydirac_0\!\sQ}\right)\cup\mydirac_1\!\sQ.
\end{align*}
This concludes our example.
\end{exmp}

In the sequel, we shall allow $\sQ\subset\RR^{d+1}$ to be an arbitrary open subset. Given a vector field
$b:\sQ\to\RR^{d+1}$, and a function $c:\sQ\to \RR$, we shall derive maximum principles for the operator,
\begin{equation}
\label{eq:Generator_parabolic}
Lu := -u_t - \tr(aD^2u) - \langle b,Du\rangle + cu,
\end{equation}
where $D^2u$ and $Du$ denote the Hessian matrix and gradient of a suitably regular function $u$ on $\sQ$ with respect to the spatial coordinates, respectively. We suppose that the coefficients, $a,b,c$, are defined on $\sQ$ in the case of maximum principles for $L$-subharmonic functions in $C^2(\sQ)$ and are measurable and defined a.e. on $\sQ$ in the case of maximum principles for those in $W^{2,d+1}_{\loc}(\sQ)$.

In older literature, $L$ in \eqref{eq:Generator_parabolic} is called a parabolic linear second-order partial differential operator with non-negative characteristic form\footnote{We refer to the definition of Ole{\u\i}nik and Radkevi{\v{c}} \cite[p. 308]{Radkevich_2009a} rather than Tricomi \cite[p. 298]{Radkevich_2009a}, which requires in addition that $a>0$ on $\sQ$, that is, $L$ is locally strictly parabolic in the interior of $\sQ$; however, in the applications we have in mind, the latter condition is often satisfied and the degeneracy is confined to a subset of $\mydirac\!\sQ$.} \cite{Oleinik_Radkevic}. We shall call $L$ \emph{boundary degenerate} when $\mydirac_0\!\sQ$ is non-empty, noting the distinction between the way we use the term `degenerate' here and the sense in which this term is used in \cite{Crandall_Ishii_Lions_1992}, where an operator which strictly parabolic is merely a particular type of degenerate parabolic operator.

We shall consider the question of uniqueness of solutions to the parabolic equation,
\begin{equation}
\label{eq:Parabolic_equation}
Lu = f \quad \hbox{(a.e.) on }\sQ,
\end{equation}
and the obstacle problem,
\begin{equation}
\label{eq:Parabolic_obstacle_problem}
\min\{Lu-f, \ u-\psi\} = 0 \quad \hbox{a.e. on }\sQ,
\end{equation}
with \emph{partial} Dirichlet boundary (and terminal) condition,
\begin{equation}
\label{eq:Parabolic_boundary_condition}
u = g \quad \hbox{on } \mydirac_1\!\sQ,
\end{equation}
for a suitably regular function $u$ on $\underline\sQ\cup \mydirac_1\!\sQ$, given a suitably regular source function $f$ on $\sQ$, boundary data $g$ on $\mydirac_1\!\sQ$, and a suitably regular obstacle function $\psi$ on $\underline\sQ\cup \mydirac_1\!\sQ$ which is compatible with $g$ in the sense that
\begin{equation}
\label{eq:Parabolic_obstacle_boundarydata_compatibility}
\psi\leq g \quad\hbox{on } \mydirac_1\!\sQ.
\end{equation}
In particular, \emph{no boundary condition} is prescribed along $\mydirac_0\!\sQ$, provided a solution $u$ is sufficiently regular up to $\mydirac_0\!\sQ$ and the coefficients of $L$ have suitable properties. We now discuss some of these properties.

\subsection{Properties of the coefficients of the parabolic operator}
\label{subsec:Properties_coefficients_parabolic}
Consider the coefficients of the parabolic operator, $L$, in \eqref{eq:Generator_parabolic}. Let $\lambda(P)$ denote the smallest eigenvalue of the matrix, $a(P)$, for each $P\in\sQ$, and let
$$
\lambda_*:\sQ\to[0,\infty)
$$
be the lower semi-continuous envelope\footnote{When $f:X\to[0,\infty)$ is a measurable function on a measure space $(X,\Sigma,\mu)$, then $f_*:X\to[0,\infty)$ is the largest lower-semicontinuous function on $X$ such that $f_*\leq f$ $\mu$-a.e. on $X$.} of the resulting least eigenvalue function,
$\lambda:\sQ\to[0,\infty]$, for $a:\sQ\to\sS^+(d)$. To achieve certain results, we may require that $a:\sQ\to\sS^+(d)$ be \emph{locally strictly parabolic on the interior}, $\sQ$, in the sense that
\begin{equation}
\label{eq:a_locally_strictly_parabolic_interior_domain}
\lambda_* > 0 \quad\hbox{on } \sQ \quad\hbox{(interior local strict parabolicity)}.
\end{equation}
Throughout the article
we shall require that\footnote{It is likely that $C^1$ would suffice, but an assumption that $\mydirac_0\!\sQ$ is a boundary portion of class $C^{1,\alpha}$ simplifies the proofs --- see \cite[Lemma B.1]{Feehan_perturbationlocalmaxima}.}
\begin{equation}
\label{eq:C1alpha_degenerate_boundary_parabolic}
\mydirac_0\!\sQ \quad\hbox{is } C^{1,\alpha},
\end{equation}
and let $\vec n$ denote the \emph{inward}-pointing unit normal vector field along $\mydirac_0\!\sQ$.

The vector field, $\vec n:\mydirac_0\!\sQ\to\RR^d$, may be extended to a tubular neighborhood $N(\mydirac_0\!\sQ)$ of $\mydirac_0\!\sQ \subset \underline\sO$, recalling that $n_0=0$ along $\mydirac_0\!\sQ$ by definition \eqref{eq:Degeneracy_locus_parabolic}. We can then split the vector field, $b:N(\mydirac_0\!\sQ)\to\RR^d$, into its normal and tangential components, with respect to the extended vector field, $\vec n:N(\mydirac_0\!\sQ)\to\RR^d$, so
\begin{equation}
\label{eq:b_splitting_parabolic}
b^\perp := \langle b, \vec n \rangle \quad\hbox{and}\quad b^\parallel := b - b^\perp \vec n \quad\hbox{on } N(\mydirac_0\!\sQ).
\end{equation}
We may require that the vector field $b^\perp$ obey one of the following conditions,
\begin{align}
\label{eq:b_perp_nonnegative_boundary_parabolic}
b^\perp &\geq 0 \quad \hbox{on } \mydirac_0\!\sQ \quad\hbox{or}
\\
\label{eq:b_perp_positive_boundary_parabolic}
\tag{\ref*{eq:b_perp_nonnegative_boundary_parabolic}$'$}
b^\perp &> 0 \quad \hbox{on } \mydirac_0\!\sQ.
\end{align}
Similarly, we may require that the function $c$ obey one of the following conditions,
\begin{align}
\label{eq:c_bounded_below}
c &\geq -K_0 \quad\hbox{(a.e.) on }\sQ, \quad\hbox{or},
\\
\label{eq:c_nonnegative_domain}
\tag{\ref*{eq:c_bounded_below}$'$}
c &\geq 0 \quad \hbox{(a.e.) on } \sQ, \quad\hbox{or}
\\
\label{eq:c_positive_lower_bound_domain}
\tag{\ref*{eq:c_bounded_below}$''$}
c &\geq c_0 \quad \hbox{(a.e.) on } \sQ.
\end{align}
When $\sQ$ is unbounded, we may couple \eqref{eq:c_bounded_below} with a requirement that
\begin{equation}
\label{eq:Domain_finite_upper_time}
\sQ\subset(-\infty,T)\times\RR^d,
\end{equation}
for some constants $K_0>0$ and $T<\infty$. We may also require that $c$ obey one of the conditions,
\begin{align}
\label{eq:c_bounded_below_boundary}
c &\geq -K_0 \quad\hbox{on } \mydirac_0\!\sQ, \quad\hbox{or}
\\
\label{eq:c_nonnegative_boundary}
\tag{\ref*{eq:c_bounded_below_boundary}$'$}
c &\geq 0 \quad \hbox{on } \mydirac_0\!\sQ, \quad\hbox{or}
\\
\label{eq:c_positive_boundary}
\tag{\ref*{eq:c_bounded_below_boundary}$''$}
c &> 0 \quad \hbox{on }  \mydirac_0\!\sQ.
\end{align}
We may require that the coefficients $b$ or $c$ are locally bounded on $\underline \sQ$, that is,
\begin{subequations}
\label{eq:bc_locally_bounded}
\begin{align}
\label{eq:b_locally_bounded}
b &\in L^\infty_{\loc}(\underline\sQ;\RR^d),
\\
\label{eq:c_locally_bounded}
c &\in L^\infty_{\loc}(\underline\sQ),
\end{align}
\end{subequations}
where we slightly abuse notation by writing $w\in L^\infty_{\loc}(\underline\sQ)$ as an abbreviation for saying that $w$ is a locally bounded function on $\underline\sQ$, irrespective of whether $w$ is measurable or everywhere-defined.

We may also require that one or more of the coefficients $a$, $b$, or $c$ be continuous along $\mydirac_0\!\sQ$,
\begin{subequations}
\label{eq:abc_continuous_degenerate_boundary}
\begin{align}
\label{eq:a_continuous_degenerate_boundary}
a &\in C(\mydirac_0\!\sQ;\sS^+(d)),
\\
\label{eq:b_continuous_degenerate_boundary}
b &\in C(\mydirac_0\!\sQ;\RR^d),
\\
\label{eq:c_continuous_degenerate_boundary}
c &\in C(\mydirac_0\!\sQ).
\end{align}
\end{subequations}
When the domain $\sQ$ is \emph{unbounded}, we will occasionally appeal to the growth condition,
\begin{equation}
\label{eq:Quadratic_growth}
\tr a(t,x) + \langle b(t,x),x\rangle \leq K(1+|x|^2), \quad\forall\, (t,x) \in \underline\sQ  \quad\hbox{(quadratic growth for $a, b$)},
\end{equation}
for some positive constant $K$.

\subsection{Application to boundary value and obstacle problems for the parabolic Heston operator}
\label{subsec:Heston}
The parabolic Heston operator \cite{Heston1993}
\begin{equation}
\label{eq:Heston_generator}
Lv := -v_t - \frac{x_2}{2}\left(v_{x_1x_1} + 2\varrho\sigma v_{x_1x_2} + \sigma^2 v_{x_2x_2}\right) - \left(r-q-\frac{x_2}{2}\right)v_{x_1} - \kappa(\theta-x_2)v_{x_2} + rv,
\end{equation}
where $v \in C^\infty(\sO_T)$ and $\sO\subseteqq\RR\times\RR_+$ and $T>0$, and provides an example of an operator of the form \eqref{eq:Generator_parabolic} and which has important applications in mathematical finance. If $Av := Lv+v_t$, then $-A$ is the generator of the $2$-dimensional Heston stochastic volatility process, $x_1$ represents the log-price of a financial asset, and $x_2$ represents its stochastic variance.

A solution to the boundary value problem \eqref{eq:Parabolic_equation}, \eqref{eq:Parabolic_boundary_condition} can be interpreted as the price of a finite-maturity European-style option with barrier condition $g\restriction (0,T)\times\partial_1\sO$ and terminal payoff $g\restriction \{T\}\times(\sO\cup\partial_1\sO)$, where $\partial_1\sO = \{x_2>0\}\cap\partial\sO$. A solution to the obstacle problem \eqref{eq:Parabolic_obstacle_problem}, \eqref{eq:Parabolic_boundary_condition} can be interpreted as the price of a finite-maturity American-style option with payoff $\psi$, barrier condition $g\restriction (0,T)\times\partial_1\sO$, and terminal payoff $g\restriction \{T\}\times(\sO\cup\partial_1\sO)$.

As we explain in Appendix \ref{sec:Fichera_and_heston}, the classical Fichera analysis of boundary conditions hinges on the sign of the Fichera function, which is in turn determined by the value of the parameter $\beta := 2\kappa\theta/\sigma^2$. As illustrated by Theorems \ref{thm:Weak_maximum_principle_C2s_bounded_domain} and \ref{thm:Weak_maximum_principle_C2s_unbounded_domain}, uniqueness of solutions in $C^2_s(\underline\sO_T)\cap C(\bar\sO_T)$ to the boundary value problem \eqref{eq:Parabolic_equation}, \eqref{eq:Parabolic_boundary_condition} does not require a boundary condition along $(0,T)\times\partial_0\sO$, where $\partial_0\sO = \Int\{\{x_2=0\}\cap\partial\sO\}$, when $\kappa\theta\geq 0$ and $r\geq 0$, irrespective of the value of $\beta>0$. The question of uniqueness of solutions in $W^{2,d+1}_{\loc}(\sO_T)\cap C^1(\underline\sO_T)\cap C(\bar\sO_T)$ to the obstacle problem \eqref{eq:Parabolic_obstacle_problem}, \eqref{eq:Parabolic_boundary_condition} is addressed in \cite{Feehan_perturbationlocalmaxima}.

The coefficients defining $L$ in \eqref{eq:Heston_generator} are constants obeying
\begin{gather}
\label{eq:Strictly_parabolic_heston}
\sigma \neq 0 \quad\hbox{and}\quad -1< \varrho < 1,
\\
\notag
\kappa > 0 \quad\hbox{and}\quad \theta > 0,
\end{gather}
while $r, q \in \RR$, though these constants are typically non-negative in financial applications. The financial and probabilistic interpretations of the preceding coefficients are provided in \cite{Heston1993}. One can show that the condition \eqref{eq:Strictly_parabolic_heston} implies that $L$ in \eqref{eq:Heston_generator} is parabolic but not strictly parabolic on $\sQ$ in the sense\footnote{The terminology is not universal.} of \cite[p. 31]{GilbargTrudinger}.

\subsection{Summary of main results and outline of our article}
\label{subsec:Summary}
We shall leave detailed statements of our main results to the body of our article and simply provide a short outline of our article here to facilitate the reader seeking a particular conclusion of interest.

Given open subsets $\sQ\subset\RR^{d+1}$ and $\Sigma \subsetneqq \mydirac\!\sQ$ and a convex cone $\fK\subset C^2(\sQ)$ (respectively, $W^{2,d+1}_{\loc}(\sQ)$), we say that an operator $L$ in \eqref{eq:Generator_parabolic} obeys the \emph{weak maximum principle property on $\sQ\cup\Sigma$ for $\fK$} (see Definition \ref{defn:Weak_maximum_principle_property_parabolic}) if whenever $u\in \fK$ obeys
$$
Lu \leq 0 \quad \hbox{(a.e.) on } \sQ \quad\hbox{and}\quad u^* \leq 0 \quad\hbox{on } \mydirac\!\sQ\less\bar\Sigma,
$$
then
$$
u \leq 0 \quad\hbox{on } \sQ.
$$
In \S \ref{sec:Application_weak_maximum_principle_property_boundary_value_problems}, regardless of whether $\sQ$ is bounded, the coefficients of $L$ obey certain growth properties, or $\fK\subset C^2(\sQ)$ or $W^{2,d+1}_{\loc}(\sQ)$, we obtain a comparison principle and a priori maximum principle estimates (Propositions \ref{prop:Comparison_principle_parabolic_boundary_value_problem} and \ref{prop:Parabolic_weak_maximum_principle_apriori_estimates}) for subsolutions, supersolutions, and solutions to the parabolic boundary value problem \eqref{eq:Parabolic_equation}, \eqref{eq:Parabolic_boundary_condition}. Theorem \ref{thm:Weak_maximum_principle_property_unbounded_functions} extends these results to the case of functions which obey a growth condition on unbounded domains.

In \S \ref{sec:Application_weak_max_principle_property_obstacleproblems}, for $\fK\subset W^{2,d+1}_{\loc}(\sQ)$, we obtain a comparison principle and a priori maximum principle estimates (Propositions \ref{prop:Comparison_principle_parabolic_obstacle_problem} and \ref{prop:Parabolic_weak_max_principle_apriori_estimates_obstacle_problem}) for supersolutions and solutions to the parabolic obstacle problem \eqref{eq:Parabolic_obstacle_problem}, \eqref{eq:Parabolic_boundary_condition}.

In \S \ref{sec:Weak_maximum_principle_C2s_subharmonic_functions}, we establish specific conditions on the coefficients $(a,b,c)$ which ensure that the operator $L$ in \eqref{eq:Generator_parabolic} has the weak maximum principle property on $\sQ\cup\Sigma$ for $\fK$ when $\Sigma = \mydirac_0\!\sQ$ and $\fK$ is the set of $u\in C^2_s(\underline\sQ)$ such that $\sup_\sQ u < \infty$. Theorem \ref{thm:Weak_maximum_principle_C2s_bounded_domain} yields the desired weak maximum principle when $\sQ$ is bounded, while Theorem \ref{thm:Weak_maximum_principle_C2s_unbounded_domain} allows $\sQ$ to be unbounded.

However, as in the classical case --- compare the proofs of the classical weak maximum principle, \cite[Theorem 8.1.4]{Krylov_LecturesHolder}, for functions in $C^2(\sQ)\cap C(\bar\sQ)$, and \cite[Corollary 7.4]{Lieberman}, for functions in $W^{2,d+1}_{\loc}(\sQ)\cap C(\bar\sQ)$ --- the establishment of a weak maximum principle for a convex cone $\fK\subset W^{2,d+1}_{\loc}(\sQ)$, when $L$ has measurable coefficients, is considerably more difficult. We establish weak maximum principles of this type in \cite{Feehan_perturbationlocalmaxima} using techniques which are quite different from those used in this article, while the development of weak maximum and comparison principles for solutions $u\in W^{1,2}_{\loc}(\sQ)$ to a variational equation or inequality defined by $L$ and suitable weighted Sobolev spaces is the subject of a separate article.

In \S \ref{sec:Strong_maximum_principle_C2s_subharmonic_functions}, we extend the methods of A. Friedman \cite{Friedman_1958}, \cite[\S 2]{FriedmanPDE} and L. Nirenberg \cite{Nirenberg_1953} to prove strong maximum principles for a boundary-degenerate, parabolic operator, $L$, where points in the degenerate-boundary portion, $\mydirac_0\!\sQ$, play the same role as points in the interior, $\sQ$. While the proofs of the weak maximum principles (Theorems \ref{thm:Weak_maximum_principle_C2s_bounded_domain} and \ref{thm:Weak_maximum_principle_C2s_unbounded_domain}) follow naturally once one has identified the right concept of degenerate-boundary regularity for an $L$-subharmonic function, $u$, the proofs of the strong maximum principles appear considerably more difficult. Although not directly used in those proofs, our approach also allows us to also establish a Hopf boundary point lemma (see Lemma \ref{lem:Degenerate_hopf_lemma_parabolic})  for a boundary-degenerate, parabolic operator, $L$. Our Hopf boundary point lemma has independent applications and, indeed, it plays an essential role in the proofs of our main results for boundary-degenerate, parabolic operators in \cite{Feehan_perturbationlocalmaxima}.

Finally, in Appendix \ref{sec:Fichera_and_heston}, we compare the maximum principles and uniqueness theorems provided by our article with those of Fichera in the case of the parabolic Heston operator, $L$, discussed in \S \ref{subsec:Heston} and show that those of Fichera are strictly weaker.

\subsection{Notation and conventions}
\label{subsec:Notation}
We let $\NN:=\left\{0,1,2,3,\ldots\right\}$ denote the set of non-negative integers. If $X$ is a subset of a topological space, we let $\bar X$ denote its closure and let $\partial X := \bar X\less X$ denote its topological boundary. For $r>0$ and $x^0\in\RR^d$, we let $B_r(x^0) := \{x\in\RR^d: |x-x^0|<r\}$ denote the open ball with center $x^0$ and radius $r$. We denote $\RR_+=(0,\infty)$ and $B_r^+(x^0) := B_r(x^0) \cap (\RR^{d-1}\times\RR_+)$ when $x^0\in \RR^{d-1}\times\{0\}\subset\RR^d$. When $x^0$ is the origin in $\RR^d$, we often denote $B_r(x^0)$ and $B_r^+(x^0)$ simply by $B_r$ and $B_r^+$ for brevity. When we wish to emphasize the dimension of a ball, we write $B^d$ for an open ball in $\RR^d$.

If $V\subset U\subset \RR^d$ are open subsets, we write $V\Subset U$ when $U$ is bounded with closure $\bar U \subset V$. By $\supp\zeta$, for any $\zeta\in C(\RR^d)$, we mean the \emph{closure} in $\RR^d$ of the set of points where $\zeta\neq 0$. We denote $x\vee y = \max\{x,y\}$ and $x\wedge y = \min\{x,y\}$, for any $x,y\in\RR$. We occasionally shall write coordinates on $\RR^d$ as $x=(x',x_d)\in\RR^{d-1}\times\RR$.

For an open subset of a topological space, $U\subset X$, we let $u^*:\bar U\to[-\infty,\infty]$ (respectively, $u_*:\bar U\to[-\infty,\infty]$) denote the upper (respectively, lower) semicontinuous envelope of a function $u:U\to[-\infty,\infty]$; when $u$ is continuous on $U$, then $u_* = u = u^*$ on $U$.

In the definition and naming of function spaces, we follow Adams \cite{Adams_1975} and alert the reader to occasional differences in definitions between R. A. Adams \cite{Adams_1975} and standard references such as D. Gilbarg and N. Trudinger \cite{GilbargTrudinger}, N. V. Krylov \cite{Krylov_LecturesHolder}, or G. Lieberman \cite{Lieberman}.

\subsection{Acknowledgments} This article was written while the author held a visiting faculty appointment in the Department of Mathematics at Columbia University, on sabbatical from Rutgers University, and completed while visiting the Max Planck Institut f\"ur Mathematik, Bonn. I am very grateful to Ioannis Karatzas and the Department of Mathematics at Columbia University, especially Panagiota Daskalopoulos and Duong Phong, and to the Max Planck Institut f\"ur Mathematik for their generous support.

\section{Applications of the weak maximum principle property to boundary value problems}
\label{sec:Application_weak_maximum_principle_property_boundary_value_problems}
We shall encounter many different situations (for example, bounded or unbounded open subsets $\sQ\subset\RR^{d+1}$, bounded or unbounded functions $u$ with prescribed growth, and so on) where a basic maximum principle holds for linear, second-order, partial differential operators $L$ in \eqref{eq:Generator_parabolic} acting on a convex cone of functions in $C^2(\sQ)$ or $W^{2,p}_{\loc}(\sQ)$. In order to unify our treatment of applications, we find it useful to isolate a key `weak maximum principle property' (Definition \ref{defn:Weak_maximum_principle_property_parabolic}) and then derive the consequences which necessarily follow in an essentially formal manner. In this section we consider applications to parabolic Dirichlet boundary value problems. After reviewing our definitions of function spaces in \S \ref{subsec:Parabolic_sobolev_embedding} and providing further interpretation of our definition of second-order boundary conditions in \S \ref{subsec:Second_order_boundary_regularity}, we proceed to the main applications in \S \ref{subsec:Weak_maximum_principle_property_apriori_estimates}, namely a comparison principle for subsolutions and supersolutions and uniqueness for solutions to the Dirichlet terminal-boundary problem (Proposition \ref{prop:Comparison_principle_parabolic_boundary_value_problem}) and a priori estimates for subsolutions, supersolutions, and solutions (Proposition \ref{prop:Parabolic_weak_maximum_principle_apriori_estimates}). Finally, we show that when an operator has the weak maximum principle property for subsolutions which are bounded above, the property may also hold for unbounded subsolutions which instead obey a growth condition (Theorem \ref{thm:Weak_maximum_principle_property_unbounded_functions}).

\subsection{Parabolic spaces of continuous functions and Sobolev spaces}
\label{subsec:Parabolic_sobolev_embedding}
For $d \geq 1$ and an open subset $\sQ\subset\RR^{d+1}$ and $p\geq 1$, we say that (following Lieberman \cite[p. 155]{Lieberman})
\begin{equation}
\label{eq:Defn_W2p}
u \in W^{2,p}(\sQ)
\end{equation}
if $u$ is a measurable function on $\sQ$ and $u$ and its weak derivatives, $u_t$ and $u_{x_i}$ and $u_{x_ix_j}$ for $1\leq i,j\leq d$, belong to $L^p(\sQ)$ and similarly define $W^{2,p}_{\loc}(\sQ)$. Here, $W^{2,p}(\sQ)$ is a \emph{parabolic Sobolev space} \cite[\S 2.2]{Krylov_LecturesSobolev}, \cite[\S 1.1]{LadyzenskajaSolonnikovUralceva}, because we only assume $u_t \in L^p(\sQ)$ and do not, in addition, assume that $u_{tt} \in L^p(\sQ)$ or $u_{tx_i} \in L^p(\sQ)$ for $1\leq i\leq d$.

We let $C(\sQ)$ denote the vector space of continuous functions on $\sQ$ and let $C(\bar \sQ)$ denote the Banach space of functions in $C(\sQ)$ which are bounded and uniformly continuous on $\sQ$, and thus have unique bounded, continuous extensions to $\bar \sQ$, with norm $\|u\|_{C(\bar \sQ)} := \sup_{\sQ}|u|$ \cite[\S 1.26]{Adams_1975}. We let $C(\underline\sQ)$ denote the vector subspace of functions $u \in C(\sQ)$ such that $u\in C(\bar\sQ')$ for every precompact open subset $\sQ'\Subset \bar \sQ$.

We shall need parabolic variants of the definitions of $C^1$ and $C^2$ functions on open subsets of $\RR^d$ in the context of elliptic problems.

\begin{defn}[Parabolic $C^1$ and $C^2$ functions]
\label{defn:C1_C2_function_parabolic}
We say that $u\in C^1(\sQ)$ (respectively, $C^1(\bar \sQ)$) if $u, u_{x_i} \in C(\sQ)$ for $1\leq i\leq d$ (respectively, $C(\bar \sQ)$; we say that $u\in C^2(\sQ)$ (respectively, $C^2(\bar \sQ)$) if $u, u_t, u_{x_i}, u_{x_ix_j} \in C(\sQ)$ for $1\leq i,j\leq d$ (respectively, $C(\bar \sQ)$.
\end{defn}

The parabolic Sobolev embedding theorem (see \cite[Lemma 2.3.3]{LadyzenskajaSolonnikovUralceva} or \cite[Theorem 3.4]{Feehan_perturbationlocalmaxima} for a restatement) implies that $W^{2,d+1}(\sQ) \subset C(\bar\sQ)$ when $\sQ=(0,T)\times\sO)$ and $\sO \subset \RR^d$ is an open subset which obeys a uniform interior cone condition. In particular, for an arbitrary open subset $\sQ \subset \RR^{d+1}$, we have $W^{2,d+1}_{\loc}(\sQ) \subset C(\sQ)$.

\subsection{Second-order boundary condition and boundary regularity}
\label{subsec:Second_order_boundary_regularity}
The second-order boundary condition \eqref{eq:Ventcel} is a property of functions in the weighted H\"older spaces, $C^{2+\alpha}_s(\underline \sQ)$, defined in \cite{DaskalHamilton1998} for functions on an open subset $\sQ\subset\RR^{d+1}$. See \cite[Proposition I.12.1]{DaskalHamilton1998}, \cite[Lemma C.1]{Feehan_perturbationlocalmaxima}, \cite[Lemma 3.1]{Feehan_Pop_mimickingdegen_pde} for further discussion. The condition \eqref{eq:Ventcel} may also be viewed as a special case of a generalized Ventcel boundary condition \cite[\S 7.1]{Taira_2004}.

If $u \in C^2_s(\underline \sQ)$, with $L$ as in \eqref{eq:Generator_parabolic}, then the second-order boundary condition \eqref{eq:Ventcel} is equivalent to
\begin{equation}
\label{eq:Equivalent_first_order_boundary_condition}
-u_t - \langle b, Du\rangle + cu \leq 0 \quad\hbox{on }\mydirac_0\!\sQ.
\end{equation}
Indeed, when we have $Lu=f$ on $\underline\sQ$ and thus equality in \eqref{eq:Equivalent_first_order_boundary_condition}, the condition \eqref{eq:Equivalent_first_order_boundary_condition} is analogous to the boundary condition proposed by S. Heston \cite[Equation (9)]{Heston1993} for the parabolic equation \eqref{eq:Parabolic_equation}: one obtains
\begin{equation}
\label{eq:Financial_engineer_first_order_boundary_condition}
-u_t - \langle b,Du\rangle + cu = f \quad\hbox{on }\mydirac_0\!\sQ,
\end{equation}
for \eqref{eq:Parabolic_equation} when $f$ is non-zero and $u \in C^2_s(\underline \sQ)$. Indeed, the condition \eqref{eq:Financial_engineer_first_order_boundary_condition} (normally when $f=0$) is often used in the numerical solution of parabolic boundary value or obstacle problems in mathematical finance \cite[Equation (22.19)]{DuffyFDM}, \cite[Equation (15)]{ZvanForsythVetzal}.

\subsection{The weak maximum principle property and a priori estimates}
\label{subsec:Weak_maximum_principle_property_apriori_estimates}
To state the weak maximum property in some generality, it is convenient to make use of the following analogue of \cite[Definition 2.8]{Feehan_maximumprinciple_v1}; compare \cite[p. 292]{Trudinger_1977}. Given a real vector space, $V$, recall that a \emph{convex cone}, $\fK\subset V$, is a subset such that if $u, v \in \fK$ and $\alpha,\beta \in \bar\RR_+$, then $\alpha u + \beta v \in \fK$.

\begin{defn}[Weak maximum principle property for $L$-subharmonic functions in $C^2(\sQ)$ or $W^{2,d+1}_{\loc}(\sQ)$]
\label{defn:Weak_maximum_principle_property_parabolic}
Let $\sQ\subset\RR^{d+1}$ be an open subset, let $\Sigma \subsetneqq \mydirac\!\sQ$ be an open subset, and let $\fK\subset C^2(\sQ)$ (respectively, $W^{2,d+1}_{\loc}(\sQ)$) be a convex cone. We say that an operator $L$ in \eqref{eq:Generator_parabolic} obeys the \emph{weak maximum principle property on $\sQ\cup\Sigma$ for $\fK$} if whenever $u\in \fK$ obeys
$$
Lu \leq 0 \quad \hbox{(a.e.) on } \sQ \quad\hbox{and}\quad u^* \leq 0 \quad\hbox{on } \mydirac\!\sQ\less\bar\Sigma,
$$
then
$$
u \leq 0 \quad\hbox{on } \sQ.
$$
\end{defn}

\begin{exmp}[Examples of the weak maximum principle property for $L$-subharmonic functions in $C^2(\sQ)$ or $W^{2,d+1}_{\loc}(\sQ)$]
\label{exmp:Examples_weak maximum principle property_C2_W2d+1_parabolic}
One can find examples of subsets $\sQ$ and $\Sigma\subseteqq\partial\sQ$, operators $L$, and cones $\fK$ yielding the weak maximum principle property in the following settings.
\begin{enumerate}
\item In \cite[Theorem 2.4]{Lieberman} (respectively, \cite[Corollaries 6.26 or 7.4]{Lieberman}), where $\sQ$ is bounded, one takes $\Sigma = \emptyset$ and $\fK = C^2(\sQ)\cap C(\bar\sQ)$ (respectively, $W^{2,d+1}_{\loc}(\sQ)\cap C(\bar\sQ)$).
\item In \cite[Theorem 8.1.4]{Krylov_LecturesHolder}, where $\sQ$ may be unbounded, one takes $\Sigma = \emptyset$ and $\fK$ to be the set of $u\in C^2(\sQ)$ such that $\sup_\sQ u < \infty$.
\item In \cite[Theorem 3.20]{Feehan_perturbationlocalmaxima} (respectively, \cite[Theorem 3.21]{Feehan_perturbationlocalmaxima}), where $\sQ$ is a bounded domain, one takes $\Sigma = \mydirac_0\!\sQ$ and $\fK = C^2(\sQ)\cap \sC^1(\underline\sQ)$ (respectively, $W^{2,d+1}_{\loc}(\sQ)\cap \sC^1(\underline\sQ)$) and $\sup_\sQ u < \infty$; here, $\sC^1(\underline\sQ)$ denotes the subset of $C(\underline\sQ)$ such that $u_t$ and $Du$ are continuous on $\underline\sQ$.
\item In Theorems \ref{thm:Weak_maximum_principle_C2s_bounded_domain} and \ref{thm:Weak_maximum_principle_C2s_unbounded_domain}, one takes $\Sigma = \mydirac_0\!\sQ$ and $\fK$ to be the set of $u\in C^2_s(\underline\sQ)$ such that $\sup_\sQ u < \infty$.
\end{enumerate}
\end{exmp}

\begin{rmk}[Weak maximum principle property for viscosity subsolutions]
\label{rmk:Example_weak maximum principle property_viscosity_parabolic}
Suppose that the coefficients of $L$ in \eqref{eq:Generator_parabolic} obey the hypotheses of \cite[Theorem 8.2 and Example 3.6]{Crandall_Ishii_Lions_1992}, so $c$ is continuous on $\sQ$ and $c\geq c_0$ on $\sQ$ for some positive constant, $c_0$; the vector field $b$ is continuous on $\sQ$ and obeys $\langle b(t,x)-b(t,y), x-y\rangle \geq -b_0|x-y|^2$ for some positive constant $b_0$; and $a = \sigma^*\sigma$ where $\sigma:\sQ\to\RR^{d\times d}$ is uniformly Lipschitz continuous. Then \cite[Theorem 8.2]{Crandall_Ishii_Lions_1992}, when $\sQ$ is bounded, implies that $L$ has the weak maximum principal property when $\Sigma = \emptyset$ and $\fK$ is the set of upper semicontinuous functions on $\bar\sQ$.
\end{rmk}

The first application, of course, of the weak maximum principle property is to settle the question of \emph{uniqueness} for solutions to the Dirichlet boundary problem.

\begin{prop}[Comparison principle for subsolutions and supersolutions to a boundary value problem]
\label{prop:Comparison_principle_parabolic_boundary_value_problem}
Let $\sQ\subset\RR^{d+1}$ be an open subset and $L$ in \eqref{eq:Generator_parabolic} have the weak maximum principle property on $\sQ\cup\Sigma$ in the sense of Definition \ref{defn:Weak_maximum_principle_property_parabolic}, for a convex cone $\fK\subset C^2(\sQ)$ (respectively, $W^{2,d+1}_{\loc}(\sQ)$)
and open subset $\Sigma\subseteqq\mydirac\!\sQ$. Suppose that $u, -v\in \fK$. If $Lu \leq Lv$ (a.e.) on $\sQ$ and $u^* \leq v_*$ on $\mydirac\!\sQ\less\bar\Sigma$, then $u \leq v$ on $\sQ$. If $Lu = Lv$ (a.e.) on $\sQ$ and $u^* = v_*$ on $\mydirac\!\sQ\less\bar\Sigma$, then $u = v$ on $\sQ$ and $u=v\in C(\sQ\cup \mydirac\!\sQ\less\bar\Sigma)$.
\end{prop}

The proof of Proposition \ref{prop:Comparison_principle_parabolic_boundary_value_problem} is identical to that of \cite[Proposition 2.16]{Feehan_maximumprinciple_v1} and so is omitted. Before we prove the weak maximum principle property for $L$ in \eqref{eq:Generator_parabolic} under suitable hypotheses on its coefficients, it is convenient to derive simple a priori estimates which Definition \ref{defn:Weak_maximum_principle_property_parabolic} and Proposition \ref{prop:Comparison_principle_parabolic_boundary_value_problem} imply.

\begin{prop}[Weak maximum principle estimates for functions in $C^2(\sQ)$ or $W^{2,d+1}_{\loc}(\sQ)$]
\label{prop:Parabolic_weak_maximum_principle_apriori_estimates}
Let $\sQ\subset\RR^{d+1}$ be an open subset and $L$ in \eqref{eq:Generator_parabolic} have the weak maximum principle property on $\sQ\cup\Sigma$ in the sense of Definition \ref{defn:Weak_maximum_principle_property_parabolic}, for a convex cone $\fK\subset C^2(\sQ)$ (respectively, $W^{2,d+1}_{\loc}(\sQ)$) containing the constant function $1$ and open subset $\Sigma\subsetneqq\mydirac\!\sQ$. Suppose that $u, -v\in \fK$.
\begin{enumerate}
\item\label{item:Subsolution_Lu_leq_zero} If $c \geq 0$ on $\sQ$ and $Lu\leq 0$ on $\sQ$, then
$$
u\leq 0 \vee \sup_{\mydirac\!\sQ\less\Sigma}u^* \quad\hbox{on } \sQ.
$$
\item\label{item:Subsolution_Lu_arb_sign} If $c\geq c_0$ on $\sQ$ for a positive constant $c_0$, then
$$
u\leq 0 \vee \frac{1}{c_0}\sup_\sQ Lu \vee \sup_{\mydirac\!\sQ\less\Sigma}u^* \quad\hbox{on } \sQ.
$$
\item\label{item:Supersolution_Lu_geq_zero} If $c \geq 0$ on $\sQ$ and $Lv\geq 0$ on $\sQ$, then
$$
v\geq 0 \wedge \inf_{\mydirac\!\sQ\less\Sigma}v_* \quad\hbox{on } \sQ.
$$
\item\label{item:Supersolution_Lu_arb_sign} If $c\geq c_0$ on $\sQ$ for a positive constant $c_0$, then
$$
v\geq 0 \wedge \frac{1}{c_0}\inf_\sQ Lv \wedge \inf_{\mydirac\!\sQ\less\Sigma}v_* \quad\hbox{on } \sQ.
$$
\item\label{item:Solution_Lu_zero} If $c \geq 0$ on $\sQ$ and $Lu=0$ on $\sQ$ and $u \in C(\sQ\cup\mydirac\!\sQ\less\Sigma)$ and $u \in \fK\cap -\fK$, then
$$
|u| \leq \|u\|_{C(\overline{\mydirac\!\sQ\less\Sigma})}  \quad\hbox{on } \sQ.
$$
\item\label{item:Solution_Lu_arb_sign} If $c\geq c_0$ on $\sQ$ for a positive constant $c_0$ and $u \in C(\sQ\cup\mydirac\!\sQ\less\Sigma)$ and $u \in \fK\cap -\fK$, then
$$
|u| \leq \frac{1}{c_0}\|Lu\|_{C(\bar \sQ)}\vee\|u\|_{C(\overline{\mydirac\!\sQ\less\Sigma})}  \quad\hbox{on } \sQ.
$$
\end{enumerate}
When $\fK\subset W^{2,d+1}_{\loc}(\sQ)$, then inequalities involving $c$ and $Lu$ or $Lv$ may hold a.e. on $\sQ$ and we write $\esssup_\sQ Lu$ and $\essinf_\sQ Lv$ and $\|Lu\|_{L^\infty(\sQ)}$ in place of $\sup_\sQ Lu$ and $\inf_\sQ Lv$ and $\|Lu\|_{C(\bar \sQ)}$.
\end{prop}

The proof of Proposition \ref{prop:Parabolic_weak_maximum_principle_apriori_estimates} is almost identical to the proof of \cite[Proposition 2.19]{Feehan_maximumprinciple_v1} and so is omitted.

The a priori estimate in Item \eqref{item:Solution_Lu_arb_sign} of Proposition \ref{prop:Parabolic_weak_maximum_principle_apriori_estimates} may be compared with its elliptic analogue \cite[Theorem 1.1.2]{Radkevich_2009a} (in the case of $C^2$ functions) and \cite[Theorem 1.5.1 and 1.5.5]{Radkevich_2009a} and \cite[Lemma 2.8]{Troianiello} (in the case of $H^1$ functions).

For a parabolic operator, $L$, the hypotheses on $c$ in Proposition \ref{prop:Parabolic_weak_maximum_principle_apriori_estimates} can usually be relaxed, as illustrated in the

\begin{lem}[Weak maximum principle estimates for functions in $C^2(\sQ)$ or $W^{2,d+1}_{\loc}(\sQ)$ when $c$ is bounded below]
\label{lem:Parabolic_weak_maximum_principle_apriori_estimates_c_bounded_below}
Let $\sQ\subset\RR^{d+1}$ be an open subset and $L$ in \eqref{eq:Generator_parabolic} have the weak maximum principle property on $\sQ\cup\Sigma$ in the sense of Definition \ref{defn:Weak_maximum_principle_property_parabolic}, for a convex cone $\fK\subset C^2(\sQ)$ (respectively, $W^{2,d+1}_{\loc}(\sQ)$) containing the constant function $1$ and open subset $\Sigma\subsetneqq\mydirac\!\sQ$. Require that $c$ and $\sQ$ obey \eqref{eq:c_bounded_below} and \eqref{eq:Domain_finite_upper_time}, respectively, for some constants $K_0>0$ and $T<\infty$. If $\fK$ is closed under multiplication by the function $e^{\lambda t}$, when $\lambda$ is a positive constant, then the estimates in Items \eqref{item:Subsolution_Lu_arb_sign}, \eqref{item:Supersolution_Lu_arb_sign}, and \eqref{item:Solution_Lu_arb_sign} in Proposition \ref{prop:Parabolic_weak_maximum_principle_apriori_estimates} hold with $1/c_0$ replaced by $e^{(K_0+1)(T-t)}$.
\end{lem}

\begin{proof}
Consider the analogue of the estimate Proposition \ref{prop:Parabolic_weak_maximum_principle_apriori_estimates} \eqref{item:Subsolution_Lu_arb_sign}. Define $u(t,x) =: e^{-\lambda t}w(t,x)$, for a positive constant $\lambda$ to be determined, and note that $w \in \fK$ by hypothesis. Because
$$
Lu = e^{-\lambda t}\left(Lw + \lambda w\right),
$$
and $Lu\leq 0$ on $\sQ$, we see that $(L + \lambda)w\leq 0$ on $\sQ$. Choose $\lambda = K_0+1$, so we have $c+\lambda\geq 1$ on $\sQ$, and apply the estimate in Proposition \ref{prop:Parabolic_weak_maximum_principle_apriori_estimates} \eqref{item:Subsolution_Lu_arb_sign}, but with $c_0$ replaced by $1$, and $u$ replaced by $w$, and $L$ replaced by $L + \lambda$ to give
$$
w \leq 0\vee \sup_\sQ  (L + \lambda)w \vee \sup_{\mydirac\!\sQ\less\Sigma}w^* \quad\hbox{on } \sQ ,
$$
that is
$$
e^{(K_0+1)t}u \leq 0\vee \sup_\sQ  e^{(K_0+1)t}Lu \vee \sup_{\mydirac\!\sQ\less\Sigma}e^{(K_0+1)t}u^* \quad\hbox{on } \sQ,
$$
yielding the estimate in
\begin{equation}
\label{eq:Finite_time_c_bounded_below_max_principle_estimate}
u \leq 0\vee e^{(K_0+1)(T-t)}\sup_\sQ  Lu \vee \sup_{\mydirac\!\sQ\less\Sigma}u^*  \quad\hbox{on } \sQ.
\end{equation}
The conclusions for this and the remaining cases follow immediately.
\end{proof}

If an operator $L$ only has the weak maximum principle property (Definition \ref{defn:Weak_maximum_principle_property_parabolic}) for functions which are \emph{bounded above}, we can obtain an extension for functions which instead obey a \emph{growth condition}.

\begin{thm}[Weak maximum principle property for unbounded functions in $C^2(\sQ)$ or $W^{2,d+1}_{\loc}(\sQ)$]
\label{thm:Weak_maximum_principle_property_unbounded_functions}
Let $\sQ\subset\RR^{d+1}$ be a possibly unbounded open subset and $\varphi\in C^2(\sQ)$ obey $0<\varphi\leq 1$ on $\sO$. Let $L$ be an operator as in \eqref{eq:Generator_parabolic} and
\begin{equation}
\label{eq:First_order_operator}
Nv := -[L, \varphi](\varphi^{-1}v), \quad\forall\, v\in C^2(\sQ),
\end{equation}
and suppose that the differential operator,
\begin{equation}
\label{eq:Defn_hatL_operator}
\widehat L := (L+N)v, \quad\forall\, v\in C^2(\sQ),
\end{equation}
has the weak maximum principle property on $\sQ\cup\Sigma$ in the sense of Definition \ref{defn:Weak_maximum_principle_property_parabolic}, for a convex cone $\fK\subset C^2(\sQ)$ (respectively, $W^{2,d+1}_{\loc}(\sQ)$) and open subset $\Sigma\subseteqq\mydirac\!\sQ$, for functions $u\in\fK$ which are \emph{bounded above}, so $\sup_\sQ u<\infty$. Then $L$ has the weak maximum principle property on $\sQ\cup\Sigma$ for functions $u\in\fK$ which obey the growth condition,
\begin{equation}
\label{eq:Upper_bound_subsolution}
u \leq C\left(1+\varphi^{-1}\right) \quad\hbox{on }\sQ.
\end{equation}
\end{thm}

The proof of Theorem \ref{thm:Weak_maximum_principle_property_unbounded_functions} is identical to that of \cite[Theorem 2.20]{Feehan_maximumprinciple_v1} and so is omitted.

\section{Applications of the weak maximum principle property to obstacle problems}
\label{sec:Application_weak_max_principle_property_obstacleproblems}
In this section, we consider the application of the weak maximum principle property to the development of a comparison principle for a supersolution and solution to an obstacle problem (Proposition \ref{prop:Comparison_principle_parabolic_obstacle_problem}) and a priori estimates for a supersolution and solution (Proposition \ref{prop:Parabolic_weak_max_principle_apriori_estimates_obstacle_problem}), when the obstacle problem is defined by a boundary-degenerate, linear, second-order, parabolic operator, $L$.

\begin{defn}[Solution and supersolution to an obstacle problem]
\label{defn:Solution_and_supersolution_obstacle_problem}
Let $\sQ\subset\RR^{d+1}$ be an open subset, $p \geq 1$, and $L$ be as in \eqref{eq:Generator_parabolic}. Given $f\in L^p_{\loc}(\sQ)$ and $\psi\in L^p_{\loc}(\sQ)$, we call $u\in W^{2,p}_{\loc}(\sQ)$ a \emph{solution} (respectively, \emph{supersolution}) to the obstacle problem \eqref{eq:Parabolic_obstacle_problem} if
$$
\min\{Lu-f,u-\psi\} = 0 \ (\geq 0) \quad \hbox{a.e. on }\sQ.
$$
Furthermore, given $g\in C(\mydirac\!\sQ\less\bar\Sigma)$ and $\psi$ also belonging to $C(\mydirac\!\sQ\less\bar\Sigma)$ and obeying the compatibility condition \eqref{eq:Parabolic_obstacle_boundarydata_compatibility}, that is, $\psi\leq g$ on $\mydirac\!\sQ\less\bar\Sigma$, we call $u$ a \emph{solution} to the obstacle problem with partial Dirichlet boundary condition if in addition $u$ belongs to $C(\mydirac\!\sQ\less\bar\Sigma)$ and is a \emph{solution} (respectively, \emph{supersolution}) to  \eqref{eq:Parabolic_boundary_condition}, so
$$
u = g \ (\geq g) \quad\hbox{on } \mydirac\!\sQ\less\bar\Sigma.
$$
\qed
\end{defn}

We first prove a comparison principle for suitably-defined supersolutions and solutions and uniqueness for solutions to the obstacle problem (Proposition \ref{prop:Comparison_principle_parabolic_obstacle_problem}) and then derive a priori maximum principle estimates for those supersolutions and solutions to obstacle problems (Proposition \ref{prop:Parabolic_weak_max_principle_apriori_estimates_obstacle_problem}).

We may compare Propositions \ref{prop:Comparison_principle_parabolic_obstacle_problem} and \ref{prop:Parabolic_weak_max_principle_apriori_estimates_obstacle_problem} with \cite[Theorems 4.5.1, 4.6.1, 4.6.6, and 4.7.4, and Corollary 4.5.2]{Rodrigues_1987} for the case of solutions and supersolutions to variational inequalities.

\begin{prop}[Comparison principle and uniqueness for $W^{2,d+1}_{\loc}$ solutions to the obstacle problem]
\label{prop:Comparison_principle_parabolic_obstacle_problem}
Let $\sQ\subset\RR^{d+1}$ be an open subset, $\fK\subset W^{2,d+1}_{\loc}(\sQ)$ be a convex cone
and $\Sigma\subseteqq\mydirac\!\sQ$ be an open subset. For every open subset $\sU\subset\sQ$, let $L$ in \eqref{eq:Generator_parabolic} have the weak maximum principle property on $\sU\cup\Sigma$ in the sense of Definition \ref{defn:Weak_maximum_principle_property_parabolic}
\footnote{Note that the weak maximum principle property hypothesis on $L$ here and in Proposition \ref{prop:Parabolic_weak_max_principle_apriori_estimates_obstacle_problem} is stronger than that in Propositions \ref{prop:Comparison_principle_parabolic_boundary_value_problem} and \ref{prop:Parabolic_weak_maximum_principle_apriori_estimates}.}.
Let $f\in L^{d+1}_{\loc}(\sQ)$ and $\psi\in L^{d+1}_{\loc}(\sQ)$.
Suppose $u\in \fK$ (respectively, $v\in -\fK$) is a solution (respectively, supersolution) to the obstacle problem,
$$
\min\{Lu-f, \ u-\psi\} = 0 \ (\geq  0) \quad\hbox{a.e. on } \sQ.
$$
If $v_*\geq u^*$ on $\mydirac\!\sQ\less\bar\Sigma$, then $v \geq u$ on $\sQ$; if $u, v$ are solutions and $v_* = u^*$ on $\mydirac\!\sQ\less\bar\Sigma$, then $u = v$ on $\sQ$.
\end{prop}

\begin{proof}
Suppose $\sU := \sQ\cap\{u>v\}$ is non-empty. Observe that by Definition \ref{defn:Parabolic_boundary} of the parabolic boundary, we have
$$
\mydirac!\sU = \left(\sQ\cap\mydirac\{u>v\}\right)\cup\left(\{u>v\}\cap \mydirac\!\sQ\right)\cup\left(\mydirac\{u>v\}\cap \mydirac\!\sQ\right),
$$
and we see that
$$
\mydirac\!\sU\less\bar\Sigma = \left(\sQ\cap\mydirac\{u>v\}\less\bar\Sigma\right)\cup\left(\{u>v\}\cap \mydirac\!\sQ\less\bar\Sigma\right)\cup\left(\mydirac\{u>v\}\less\bar\Sigma\cap \mydirac\!\sQ\less\bar\Sigma\right).
$$
Because $u\leq v$ on $\mydirac\!\sQ\less\bar\Sigma$ (non-empty by hypothesis) and $u=v$ on $\partial\sU$ (topological boundary), so $u=v$ on $(\sQ\cup\mydirac\!\sQ\less\bar\Sigma)\cap\mydirac\{u>v\}\less\bar\Sigma$, we must have
$$
u-v \leq 0 \quad\hbox{on } \mydirac\!\sU\less\bar\Sigma.
$$
We have $u-v\in \fK$ by hypothesis. Moreover, $L(u-v)\leq 0$ a.e on $\sU$, so $u-v\leq 0$ on $\sU$ since $L$ has the weak maximum principle property on $\sU\cup(\Sigma\cap \mydirac\!\sU)$ for $\fK\cap W^{2,d+1}_{\loc}(\sU)$ in the sense of Definition \ref{defn:Weak_maximum_principle_property_parabolic}, contradicting our assertion that $\sU$ is non-empty. Hence, $u\leq v$ on $ \sQ$.

If both $u$ and $v$ are solutions to the obstacle problem then, since any solution is also a supersolution by Definition \ref{defn:Solution_and_supersolution_obstacle_problem}, we may reverse the roles of $u$ and $v$ in the preceding argument to give $v\leq u$ on $\sQ$ and thus $u=v$ on $\sQ$.
\end{proof}

We then have the

\begin{prop}[Weak maximum principle and a priori estimates for supersolutions and solutions to obstacle problems]
\label{prop:Parabolic_weak_max_principle_apriori_estimates_obstacle_problem}
Let $\sQ\subset\RR^{d+1}$ be an open subset, $\fK\subset W^{2,d+1}_{\loc}(\sQ)$ be a convex cone containing the constant function $1$, and $\Sigma\subsetneqq\mydirac\!\sQ$ be an open subset. For every open subset $\sU\subset\sO$, let $L$ in \eqref{eq:Generator_parabolic} have the weak maximum principle property on $\sU\cup\Sigma$ in the sense of
Proposition \ref{prop:Comparison_principle_parabolic_obstacle_problem}. Assume that $c \geq 0$ a.e. on $\sQ$. Let $f\in L^{d+1}_{\loc}(\sQ)$, and $g\in C(\mydirac\!\sQ\less\bar\Sigma)$, and $\psi\in C(\sQ\cup\mydirac\!\sQ\less\bar\Sigma)$ with $\psi\leq g$ on $\mydirac\!\sQ\less\bar\Sigma$.
Suppose $u\in \fK\cap -\fK$ is a solution and $v\in -\fK$ is a supersolution to the obstacle problem in the sense of Definition \ref{defn:Solution_and_supersolution_obstacle_problem} for $f$ and $g$ and $\psi$.
\begin{enumerate}
\item \label{item:Obstacle_supersolution_f_geq_zero} If $f\geq 0$ a.e. on $\sQ$, then
$$
v\geq 0 \wedge \inf_{\mydirac\!\sQ\less\Sigma}g \quad\hbox{on }\sQ.
$$
\item\label{item:Obstacle_supersolution_f_arb_sign} If there is a constant $c_0>0$ such that $c\geq c_0$ a.e. on $\sQ$, then
$$
v\geq 0 \wedge \frac{1}{c_0}\essinf_\sQ f \wedge \inf_{\mydirac\!\sQ\less\Sigma}g \quad\hbox{on }\sQ.
$$
\item \label{item:Obstacle_solution_f_leq_zero} If $f\leq 0$ a.e on $\sQ$, then
$$
u\leq 0 \vee \sup_{\mydirac\!\sQ\less\Sigma}g \vee \sup_\sQ\psi \quad\hbox{on }\sQ.
$$
\item\label{item:Obstacle_solution_f_arb_sign} If $c\geq c_0$ a.e. on $\sQ$, then
$$
u\leq 0 \vee \frac{1}{c_0}\esssup_\sQ f \vee \sup_{\mydirac\!\sQ\less\Sigma}g \vee \sup_\sQ\psi \quad\hbox{on }\sQ.
$$
\item\label{item:Obstacle_comparison} If $u_1$ and $u_2$ are solutions, respectively, for $f_1\geq f_2$ a.e. on $\sQ$ and $\psi_1\geq \psi_2$ on $\sQ$, and $g_1\geq g_2$ on $\mydirac\!\sQ\less\bar\Sigma$, then
$$
u_1\geq u_2 \quad\hbox{on }\sQ.
$$
\item\label{item:Obstacle_stability} If $u_i$ is a solution for $f_i,\psi_i$ on $\sQ$ and $g_i$ on $\mydirac\!\sQ\less\bar\Sigma$ with $\psi_i\leq g_i$ on $\mydirac\!\sQ\less\bar\Sigma$ for $i=1,2$, and $c\geq c_0$ a.e. on $\sQ$, then
$$
|u_1-u_2| \leq \frac{1}{c_0}\|f_1-f_2\|_{L^\infty(\sQ)} \vee \|g_1-g_2\|_{C(\overline{\mydirac\!\sQ\less\Sigma})} \vee \|\psi_1-\psi_2\|_{C(\bar \sQ)} \quad\hbox{on }\sQ,
$$
and if $f_1=f_2$ and $c\geq 0$ a.e. on $\sQ$, then
$$
|u_1-u_2| \leq \|g_1-g_2\|_{C(\overline{\mydirac\!\sQ\less\Sigma})} \vee \|\psi_1-\psi_2\|_{C(\bar \sQ)}  \quad\hbox{on }\sQ.
$$
\end{enumerate}
\end{prop}

The proof of Proposition \ref{prop:Parabolic_weak_max_principle_apriori_estimates_obstacle_problem} is almost identical to that of its elliptic analogue, \cite[Proposition 3.5]{Feehan_maximumprinciple_v1}, except that the role of the comparison principle \cite[Proposition 3.3]{Feehan_maximumprinciple_v1} is replaced by that of Proposition \ref{prop:Comparison_principle_parabolic_obstacle_problem} and so we omit the proof.

We have the following analogue of Lemma \ref{lem:Parabolic_weak_maximum_principle_apriori_estimates_c_bounded_below}.

\begin{lem}[Weak maximum principle and a priori estimates for supersolutions and solutions to obstacle problems when $c$ is bounded below]
\label{lem:Parabolic_weak_max_principle_apriori_estimates_c_bounded_below_obstacle_problem}
Let $\sQ\subset\RR^{d+1}$ be an open subset and $L$ in \eqref{eq:Generator_parabolic} have the weak maximum principle property on $\sQ\cup\Sigma$ in the sense of Proposition \ref{prop:Comparison_principle_parabolic_obstacle_problem}, for a convex cone $\fK\subset W^{2,d+1}_{\loc}(\sQ)$ containing the constant function $1$ and open subset $\Sigma\subsetneqq\mydirac\!\sQ$. Require that $c$ and $\sQ$ obey \eqref{eq:c_bounded_below} and \eqref{eq:Domain_finite_upper_time}, respectively, for some constants $K_0>0$ and $T<\infty$. If $\fK$ is closed under multiplication by the function $e^{\lambda t}$, when $\lambda$ is a positive constant, then the estimates in Items \eqref{item:Obstacle_supersolution_f_arb_sign}, \eqref{item:Obstacle_solution_f_arb_sign}, and \eqref{item:Obstacle_stability} in Proposition \ref{prop:Parabolic_weak_max_principle_apriori_estimates_obstacle_problem} hold with $1/c_0$ replaced by $e^{(K_0+1)(T-t)}$.
\end{lem}

\section{Weak maximum principle for $L$-subharmonic functions in $C^2_s$}
\label{sec:Weak_maximum_principle_C2s_subharmonic_functions}
Having considered applications of the weak maximum principle property (Definition \ref{defn:Weak_maximum_principle_property_parabolic}) to Dirichlet boundary value problems in \S \ref{sec:Application_weak_maximum_principle_property_boundary_value_problems} and obstacle problems in \S \ref{sec:Application_weak_max_principle_property_obstacleproblems}, we now establish conditions under which the operator $L$ in \eqref{eq:Generator_parabolic} has the weak maximum principle property for $C^2_s$ functions on $\underline\sQ$. In \S \ref{subsec:Weak_max_principle_bounded_C2s_function_bounded_domain}, we establish a weak maximum principle for bounded $C^2_s$ functions on bounded domains (Theorem \ref{thm:Weak_maximum_principle_C2s_bounded_domain}), while in \S \ref{subsec:WeakMaxPrincipleBVProbSmoothUnboundedDomain}, we extend this result to the case of bounded $C^2_s$ functions on unbounded domains (Theorem \ref{thm:Weak_maximum_principle_C2s_unbounded_domain}).

Our weak maximum principle (Theorems \ref{thm:Weak_maximum_principle_C2s_bounded_domain} and \ref{thm:Weak_maximum_principle_C2s_unbounded_domain}) differs in several aspects from \cite[Theorem 1.1.2]{Radkevich_2009a}, some of which may appear subtle at first glance but which are nonetheless important for applications:
\begin{enumerate}
\item The function $u$ is \emph{not} required to be belong to $C^2(\underline\sQ)\cap C(\bar\sQ)$, but rather $C^2_s(\underline\sQ)\cap C(\sQ) $ and obey $\sup_\sQ u<\infty$;
\item The open subset $\sQ\subset\RR^{d+1}$ is allowed to be \emph{unbounded}; and
\item The coefficients of the partial differential operator $L$ in \eqref{eq:Generator_parabolic} are allowed to be \emph{unbounded}.
\end{enumerate}
The significance of these points is illustrated further by the example of the Heston operator discussed in \S \ref{subsec:Heston} and Appendix \ref{sec:Fichera_and_heston}.

\subsection{Weak maximum principle on bounded domains for bounded functions in $C^2_s$}
\label{subsec:Weak_max_principle_bounded_C2s_function_bounded_domain}
We begin with the case of bounded domains and adapt the proofs of \cite[Theorem 8.1.2 and Corollary 8.1.3]{Krylov_LecturesHolder}, \cite[Lemma 2.3 and Theorem 2.4]{Lieberman}; see also \cite[Theorem I.3.1]{DaskalHamilton1998}, \cite[Lemma 3.4]{Feehan_Pop_mimickingdegen_pde}. Note that while the hypotheses (and proof) of the classical weak maximum principle given by L. C. Evans in \cite[Theorems 7.1.8 and 7.1.9]{Evans} and Lieberman in \cite[Lemmas 2.1 and 2.2]{Lieberman} require a strong non-parabolic boundary regularity condition, $u\in C^2([0,T)\times\sO)\cap C([0,T]\times\bar\sO)$ (keeping in mind our convention of considering terminal rather than initial value problems, unlike in \cite{Evans} or \cite{Lieberman}), the hypotheses (and proof) of the classical weak maximum principle given by Krylov in \cite[Theorem 8.1.2 and Corollary 8.1.3]{Krylov_LecturesHolder} or Lieberman in \cite[Lemma 2.3 and Theorem 2.4]{Lieberman} only require that $u\in C^2((0,T)\times\sO)\cap C([0,T]\times\bar\sO)$.

Because of our sign convention, Lieberman's condition \cite[Equation (2.4a)]{Lieberman} is equivalent to our $c\leq -K_0$ on $\sQ$, while in this article we consider a special case of Lieberman's oblique boundary condition, $Mu\geq 0$ on $\mydirac_1\!\sQ$, where $M$ is given by \cite[Equation (2.3)]{Lieberman},
$$
Mu := \langle\beta,\partial u\rangle + \beta^0u,
$$
and $\partial u := (u_t,Du)$ denotes the full space-time gradient, $\beta:\mydirac_1\!\sQ\to\RR^{d+1}$ is a vector field which points into $\sQ$ (in the sense of \cite[p. 8]{Lieberman}), and $\beta^0:\mydirac_1\!\sQ\to\RR$ is a scalar function. Thus, a boundary condition, $u\leq 0$ on $\mydirac_1\!\sQ$, corresponds to taking $\beta\equiv 0$ and choosing $\beta^0=-1$ in \cite[Equation (2.4b)]{Lieberman}.

\begin{thm}[Weak maximum principle on bounded domains for bounded functions in $C^2_s$]
\label{thm:Weak_maximum_principle_C2s_bounded_domain}
Let $\sQ\subset\RR^{d+1}$ be a \emph{bounded} open subset and assume that the coefficients of $L$ obey \eqref{eq:b_perp_nonnegative_boundary_parabolic}, \eqref{eq:c_bounded_below}, and \eqref{eq:c_bounded_below_boundary}. Suppose that $u\in C^2_s(\underline\sQ)$ obeys $\sup_\sQ u<\infty$. If $Lu\leq 0$ on $\sQ$ and $u^*\leq 0$ on $\mydirac_1\!\sQ$, then $u\leq 0$ on $\sQ$. In particular, $L$ has the weak maximum principle property on $\underline\sQ$ in the sense of Definition \ref{defn:Weak_maximum_principle_property_parabolic}, when $\Sigma = \mydirac_0\!\sQ$ and $\fK$ is the set of $u\in C^2_s(\underline\sQ)$ such that $\sup_\sQ u < \infty$.
\end{thm}

\begin{proof}
For the reasons noted prior to the statement of Theorem \ref{thm:Weak_maximum_principle_C2s_bounded_domain}, we shall adapt the proof of \cite[Lemma 2.3]{Lieberman} due to Lieberman; when $\sQ = (0,T)\times\sO$, the argument is simpler and an adaptation of the proof \cite[Theorem 8.1.2]{Krylov_LecturesHolder} due to Krylov would suffice.

By hypotheses \eqref{eq:c_bounded_below} and \eqref{eq:c_bounded_below_boundary}, we have $c\geq -K_0$ on $\underline\sQ$, for some positive constant $K_0$, and because $\sQ$ is bounded by hypothesis, we may suppose that $\sQ$ obeys \eqref{eq:Domain_finite_upper_time}, that is, $\sQ\subset(-\infty,T)\times\RR^d$ for some positive constant, $T$. Therefore, by the method of proof of Lemma \ref{lem:Parabolic_weak_maximum_principle_apriori_estimates_c_bounded_below}, we may assume without loss of generality that $c\geq 1$ on $\underline\sQ$.

Since $u^*$ is upper semicontinuous on the compact set $\bar\sQ$, there exists a point $P^0 \in \bar\sQ$ such that $u^*(P^0)=\sup_\sQ u$. Suppose $u^*(P^0)>0$, in which case the boundary condition, $u^*\leq 0$ on $\mydirac_1\!\sQ$, implies that $P^0 \notin \overline{\mydirac_1\!\sQ}$ and thus we must have $P^0 \in \sQ$ or $\mydirac_0\!\sQ$ or $\partial\sQ\less\mydirac\!\sQ$. Note that our hypothesis, $Lu\leq 0$ on $\sQ$, is equivalent to $Lu\leq 0$ on $\underline\sQ$ by Definition \ref{defn:Second-order_boundary_regularity} since $u\in C^2_s(\underline\sQ)$.

If $P^0 \in\sQ$, then $u^*(P^0)=u(P^0)$ and $a(P^0)\geq 0$ by \eqref{eq:a_nonnegative}, and calculus yields $D^2u(P^0)\leq 0$ and $Du(P^0)=0$ and $u_t(P^0)=0$, so that
$$
Lu(P^0) = -u_t(P^0) - \tr(a(P^0)D^2u(P^0)) - \langle b(P^0), Du(P^0)\rangle + c(P^0)u(P^0) > 0,
$$
contradicting our assumption that $Lu\leq 0$ on $\underline\sQ$. Similarly, if $P^0 \in \mydirac_0\!\sQ$, then we have $\tr(a(P^0)D^2u(P^0))=0$ by Definition \ref{defn:Second-order_boundary_regularity} and $D_{\vec\tau}u(P^0)=0$ (for any tangential vector $\vec\tau(P^0)$) and $D_{\vec n}u(P^0)\geq 0$ (where $\vec n(P^0)$ is the inward-pointing normal vector) by calculus and $u_t(P^0)=0$ by \eqref{eq:Degeneracy_locus_parabolic}, and $b^\perp(P^0) \geq 0$ by \eqref{eq:b_perp_nonnegative_boundary_parabolic}, so we obtain
\begin{align*}
Lu(P^0) &= -u_t(P^0) - \tr(a(P^0)D^2u(P^0)) - \langle b(P^0), Du(P^0)\rangle + c(P^0)u(P^0)
\\
&= b^\perp(P^0) D_{\vec n}u(P^0) + c(P^0)u(P^0) > 0,
\end{align*}
again contradicting our assumption that $Lu\leq 0$ on $\underline\sQ$.

It remains to consider the case $P^0\in\partial\sQ\less\mydirac\!\sQ$. But the proof of \cite[Lemma 2.3]{Lieberman} (or the simpler proof of \cite[Theorem 8.1.2]{Krylov_LecturesHolder} when $\sQ$ is a cylinder) now applies to show that this remaining case cannot occur. This completes the proof of Theorem \ref{thm:Weak_maximum_principle_C2s_bounded_domain}.
\end{proof}

\begin{rmk}[Application to the parabolic Heston operator]
\label{rmk:Parabolic_heston_C2s_bounded}
The hypotheses of Theorem \ref{thm:Weak_maximum_principle_C2s_bounded_domain} on the coefficients of $L$ are obeyed in the case of the parabolic Heston operator \eqref{eq:Heston_generator}, where $d=2$ and  $\mydirac_0\!\sQ\subset(0,T)\times\RR\times\{0\}$. Since $\vec n=e_2$ on $\mydirac_0\!\sQ$, we have $b^2(t,x_1,x_2) = \kappa(\theta-x_2)$ and so the condition \eqref{eq:b_perp_nonnegative_boundary_parabolic} is obeyed as $b^\perp(t,x_1,0) = b^2(t,x_1,0) = \kappa\theta > 0$ on $(0,T)\times\RR\times\{0\}$, while $r\geq 0$ and thus \eqref{eq:c_nonnegative_domain}, \eqref{eq:c_nonnegative_boundary} are obeyed.
\end{rmk}

\subsection{Weak maximum principle for bounded functions in $C^2_s$ on unbounded domains}
\label{subsec:WeakMaxPrincipleBVProbSmoothUnboundedDomain}
Next, we extend Theorem \ref{thm:Weak_maximum_principle_C2s_bounded_domain} to the case of bounded $C^2_s$ functions on \emph{unbounded} open subsets. Our proof of the weak maximum principle for bounded $C^2_s$ functions on unbounded open subsets and parabolic operators with non-negative characteristic form is a modification of the proof of \cite[Theorem 5.3]{Feehan_maximumprinciple_v1} (the analogous statement for elliptic operators) and \cite[Lemma 3.4]{Feehan_Pop_mimickingdegen_pde} (where $Q=(0,T)\times\RR^{d-1}\times\RR_+$ and $a(t,x)=x_d\bar a(t,x)$ with $\bar a(t,x)$ strictly elliptic), which are based in turn on the proofs of \cite[Theorem 8.1.4 and Exercise 8.1.22]{Krylov_LecturesHolder}.

\begin{thm}[Weak maximum principle for bounded functions in $C^2_s$ on unbounded domains]
\label{thm:Weak_maximum_principle_C2s_unbounded_domain}
Let $\sQ\subseteqq\RR^{d+1}$ be a possibly unbounded open subset and assume that the coefficients $a$, $b$ of $L$ in \eqref{eq:Generator_parabolic} obey \eqref{eq:b_perp_nonnegative_boundary_parabolic} and \eqref{eq:Quadratic_growth}. Require that either $c$ obeys \eqref{eq:c_positive_lower_bound_domain} or that $c$ and $\sQ$ obey \eqref{eq:c_bounded_below} and \eqref{eq:Domain_finite_upper_time}, respectively. If $u\in C^2_s(\underline\sQ)$ obeys $\sup_\sQ u<\infty$ and $u^*\leq 0$ on $\mydirac_1\!\sQ$, then
$$
u \leq 0\vee\frac{1}{c_0}\sup_\sQ  Lu \quad\hbox{on } \sQ ,
$$
when $c$ obeys \eqref{eq:c_positive_lower_bound_domain}, or
$$
u(t,x) \leq 0\vee e^{(K_0+1)(T-t)}\sup_\sQ  Lu, \quad \forall\,(t,x)\in\sQ,
$$
when $c$ and $\sQ$ obey \eqref{eq:c_bounded_below} and \eqref{eq:Domain_finite_upper_time}. In particular, $L$ has the weak maximum principle property on $\underline\sQ$ in the sense of Definition \ref{defn:Weak_maximum_principle_property_parabolic}, when $\Sigma = \mydirac_0\!\sQ$ and $\fK$ is the set of $u\in C^2_s(\underline\sQ)$ such that $\sup_\sQ u < \infty$.
\end{thm}

\begin{proof}
It suffices to consider the case where $c$ obeys \eqref{eq:c_positive_lower_bound_domain}, as the case where $c$ and $\sQ$ obey \eqref{eq:c_bounded_below} and \eqref{eq:Domain_finite_upper_time} follows from the proof of Lemma \ref{lem:Parabolic_weak_maximum_principle_apriori_estimates_c_bounded_below}. When $c$ obeys \eqref{eq:c_positive_lower_bound_domain}, the argument is almost identical to the proof of the corresponding elliptic weak maximum principle for bounded $C^2_s$ functions on unbounded open subsets \cite[Theorem 5.3]{Feehan_maximumprinciple_v1} and so we omit the details.
\end{proof}

\begin{rmk}[Application to the parabolic Heston operator]
\label{rmk:EllipticHestonC2UnboundedDomain}
The hypotheses in Theorem \ref{thm:Weak_maximum_principle_C2s_unbounded_domain} on $L$ are obeyed in the case of the parabolic Heston operator \eqref{eq:Heston_generator} with
$$
a(t,x) = \frac{x_2}{2}\begin{pmatrix}1 & \rho\sigma \\ \rho\sigma & \sigma^2 \end{pmatrix},
\quad
b(t,x) = \begin{pmatrix} r-q-\displaystyle\frac{x_2}{2} & \kappa(\theta-x_2)\end{pmatrix},
\quad\hbox{and}\quad c(t,x) = r,
$$
where the constant coefficients are as described in \S \ref{subsec:Heston}.
\end{rmk}

\section{Strong maximum principle for $L$-subharmonic functions in $C^2_s$}
\label{sec:Strong_maximum_principle_C2s_subharmonic_functions}
Our goal in this section is to prove an analogue of the classical strong maximum principle for linear parabolic second-order operators \cite[Theorem 2.1]{FriedmanPDE} by adapting the argument in \cite{FriedmanPDE} which is in turn based on the proof due to Nirenberg \cite{Nirenberg_1953} and refinements due to Friedman \cite{Friedman_1958}. In particular, we wish to avoid an appeal to the Harnack inequality, as employed by Evans in his exposition of the proof of the classical strong maximum principle \cite[Theorems 7.1.11 and 7.1.12]{Evans}. An alternative approach to a proof of the classical strong maximum principle is provided by Lieberman in \cite[Theorem 2.7]{Lieberman}. See also \cite{Ciomaga_2012} for an interesting extension due to A. Ciomaga for the case of viscosity solutions to fully nonlinear parabolic partial integro-differential equations.

As in \cite[\S 2.1]{FriedmanPDE}, we shall consider an operator of the form \eqref{eq:Generator_parabolic}, but note that our convention for the sign of $L$ is \emph{opposite} to that of \cite[\S 2.1]{FriedmanPDE}. Our proof of the strong maximum principle (Theorem \ref{thm:Friedman_theorem_2-1}) follows the pattern of proof \cite[Theorem 2.1]{FriedmanPDE}, except that points in $\mydirac_0\!\sQ$ are regarded as `interior'. While we shall assume that the coefficients $b, c$ of $L$ obey \eqref{eq:bc_locally_bounded}, that is, are locally bounded on $\sQ$, we shall not require $a,b,c$ to be continuous on $\sQ$ as assumed in \cite[\S 2.1]{FriedmanPDE}. (Friedman notes \cite[p. 34]{FriedmanPDE} that the results of \cite[\S 2.1]{FriedmanPDE} extend to the case where $a,b,c$ are bounded on $\sQ$.) However, we require that the coefficients $a, b$ obey \eqref{eq:a_continuous_degenerate_boundary} and \eqref{eq:b_continuous_degenerate_boundary}, that is, are continuous on $\mydirac_0\!\sQ$. Before we can state and prove the strong maximum principle for the parabolic operator, we adapt the following \emph{notational conventions} of \cite[p. 34]{FriedmanPDE}.

\begin{defn}[Connected subsets of $\sQ\subset\RR^{d+1}$]
\label{defn:Connected_subsets_parabolic}
For any point $P^0=(t^0,x^0) \in \underline \sQ$, we denote by $S(P^0)$ the set of all points $P \in \underline \sQ$ which can be connected to $P^0$ by a simple continuous curve in $\underline \sQ$ along which the time coordinate is non-\emph{increasing}\footnote{In \cite[p. 34]{FriedmanPDE}, the time coordinate is required to be non-\emph{decreasing}, consistent with Friedman's convention of considering an initial value problem rather than the convention of considering a terminal value problem in this article and \cite{Bensoussan_Lions}.}
from $P$ to $P^0$. By $C(P^0)$ we denote the connected component of $\underline \sQ\cap\{t=t^0\}$ which contains $P^0$.
\end{defn}

Clearly, $C(P^0) \subset S(P^0)$. Since $C(P^0)$ is a connected component of $\underline \sQ\cap\{t=t^0\}$, it is necessarily a closed subset of $\underline \sQ\cap\{t=t^0\}$ and, if the number of connected components is finite, then it is also an open subset.

\begin{exmp}[$S(P^0)$ and $C(P^0)$ when $\sQ$ is a parabolic cylinder]
\label{exmp:Connected_subsets_parabolic_cylinder}
Suppose, as in Example \ref{exmp:Boundary_cylinder_parabolic}, that $\sQ = (0,T)\times\sO = \sO_T$ for some spatial domain $\sO\subseteqq\RR^d$ and $T>0$. For any $P^0 = (t^0,x^0) \in \underline \sQ$, Definition \ref{defn:Connected_subsets_parabolic} yields
$$
S(P^0) = (0,t_0]\times\underline\sO \quad\hbox{and}\quad C(P^0) = \{t^0\}\times\underline\sO,
$$
and so
$$
S(P^0) = \left((0,t_0)\times\underline\sO\right) \cup \left(\{t^0\}\times\underline\sO\right) = \underline\sO_{t^0} \cup C(P^0).
$$
This concludes our example.
\end{exmp}

We now proceed as in \cite[\S 2.1]{FriedmanPDE} and begin with the following analogue\footnote{We omit the assumption that $L$ is locally strictly parabolic on $\sQ$ (see \cite[Assumption (A), p. 34]{FriedmanPDE}) and the assumption  (see \cite[pp. 33--34]{FriedmanPDE}) that the coefficients of $L$ are continuous or bounded, since these conditions are not needed for the proof of \cite[Lemma 2.1]{FriedmanPDE}.} of \cite[Lemma 2.1]{FriedmanPDE}.

\begin{lem}[A special case of the strong maximum principle]
\label{lem:Friedman_lemma_2-1}
Let $\sQ\subset\RR^{d+1}$ be an open subset and assume that the coefficients of $L$ obey
\eqref{eq:b_perp_nonnegative_boundary_parabolic}, \eqref{eq:c_nonnegative_domain}, and \eqref{eq:c_nonnegative_boundary}. If $u \in C^2_s(\underline\sQ)$ obeys either $Lu<0$ on $\underline\sQ$ or $Lu\leq 0$ on $\sQ$ and $c>0$ on $\underline\sQ$, then $u$ cannot have a positive maximum in $\underline\sQ$.
\end{lem}

\begin{proof}
The proof by contradiction is the same as that of \cite[Lemma 2.1]{FriedmanPDE} when one assumes that $u$ has a positive maximum in $\sQ$ and is similar to the argument in Steps 1 and 2 of the proof of \cite[Theorem 5.1]{Feehan_maximumprinciple_v1} if one assumes that $u$ has a positive maximum in $\mydirac_0\!\sQ$. For completeness and because there some differences between the elliptic case discussed in \cite[Theorem 5.1]{Feehan_maximumprinciple_v1} and the parabolic case discussed here, we provide the details when $u$ has a positive maximum at a point $P^0\in\mydirac_0\!\sQ$ in the degenerate boundary.

We have $D_{\vec \tau}u(P^0) = 0$ for any spatial direction $\vec\tau\in\RR^d$ which is tangential to $\mydirac_0\!\sQ$ at $P^0$ and $u_t(P^0)=0$ since $n_0(P^0)=0$ and $e_0$ is tangential to $\mydirac_0\!\sQ$ at $P^0$ by definition \eqref{eq:Degeneracy_locus_parabolic} of $\mydirac_0\!\sQ$. Moreover, $D_{\vec n}u(P^0) \leq 0$ since $P^0\in\mydirac_0\!\sQ$ is a local maximum and $u\in C^1(\underline\sQ)$ by Definition \ref{defn:Second-order_boundary_regularity}. Thus,
\begin{align*}
Lu(P^0) &= -u_t(P^0) -\tr(a(P^0)D^2u(P^0)) - \langle b(P^0), Du(P^0)\rangle + c(P^0)u(P^0)
\\
&= -b^\perp(P^0) D_{\vec n}u(P^0) + c(P^0)u(P^0)  \quad\hbox{(by Definition \ref{defn:Second-order_boundary_regularity})}
\\
&\geq c(P^0)u(P^0) \quad\hbox{(by \eqref{eq:b_perp_nonnegative_boundary_parabolic})}.
\end{align*}
Suppose that $Lu \leq 0$ on $\sQ$ and $c>0$ on $\mydirac_0\!\sQ$. By continuity of $Lu$ on $\underline\sQ$ via Definition \ref{defn:Second-order_boundary_regularity}, we obtain $Lu \leq 0$ on $\underline\sQ$. If $u(P^0)>0$, we would have $c(P^0)u(P^0) > 0$ and thus $Lu(P^0)>0$, a contradiction.

Suppose that $Lu < 0$ on $\underline\sQ$ and $c\geq 0$ on $\mydirac_0\!\sQ$. If $u(P^0)>0$, we would obtain $c(P^0)u(P^0) \geq 0$ and thus $Lu(P^0)\geq 0$, again a contradiction. This completes the proof of the lemma.
\end{proof}

Next, we have an analogue of \cite[Lemma 2.2]{FriedmanPDE}, replacing the role of the Hopf boundary point lemmas \cite[Lemma 4.1]{Feehan_maximumprinciple_v1}, \cite[Lemma 3.4]{GilbargTrudinger} in the case of the proof of the strong maximum principle for an elliptic, linear, second-order partial differential operator with nonnegative characteristic form.

\begin{figure}
 \centering
 \begin{picture}(460,180)(0,0)
 \put(20,0){\includegraphics[scale=0.5]{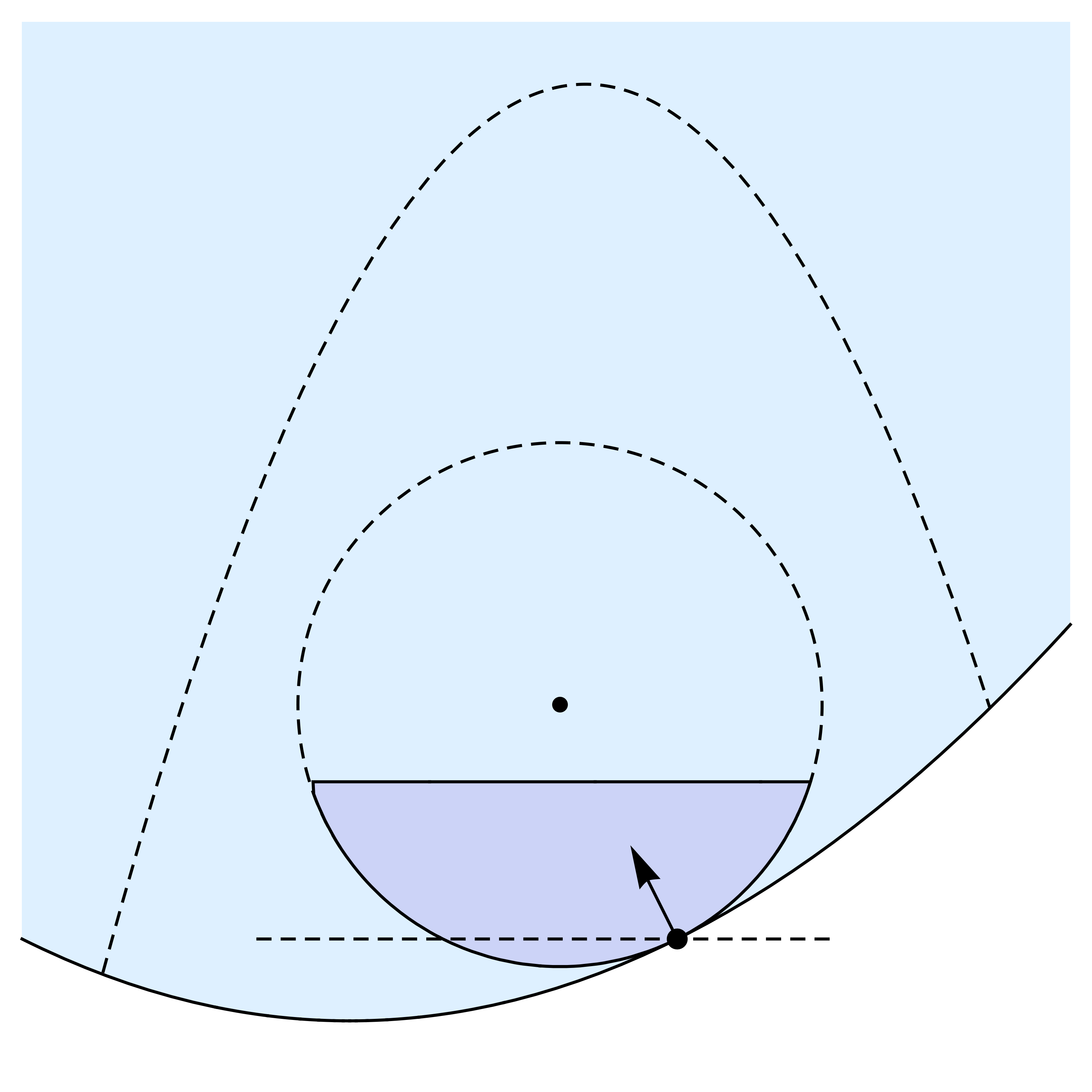}}
 \put(50,140){$\scriptstyle\sQ$}
 \put(50,4){$\scriptstyle \mydirac_0\!\sQ$}
 \put(130,16){$\scriptstyle O$}
 \put(162,23){$\scriptstyle x_d=0$}
 \put(102,140){$\scriptstyle N$}
 \put(102,90){$\scriptstyle E$}
 \put(105,65){$\scriptstyle P^*$}
 \put(90,37){$\scriptstyle D$}
 \put(127,40){$\scriptstyle \vec n$}
 \put(260,0){\includegraphics[scale=0.5]{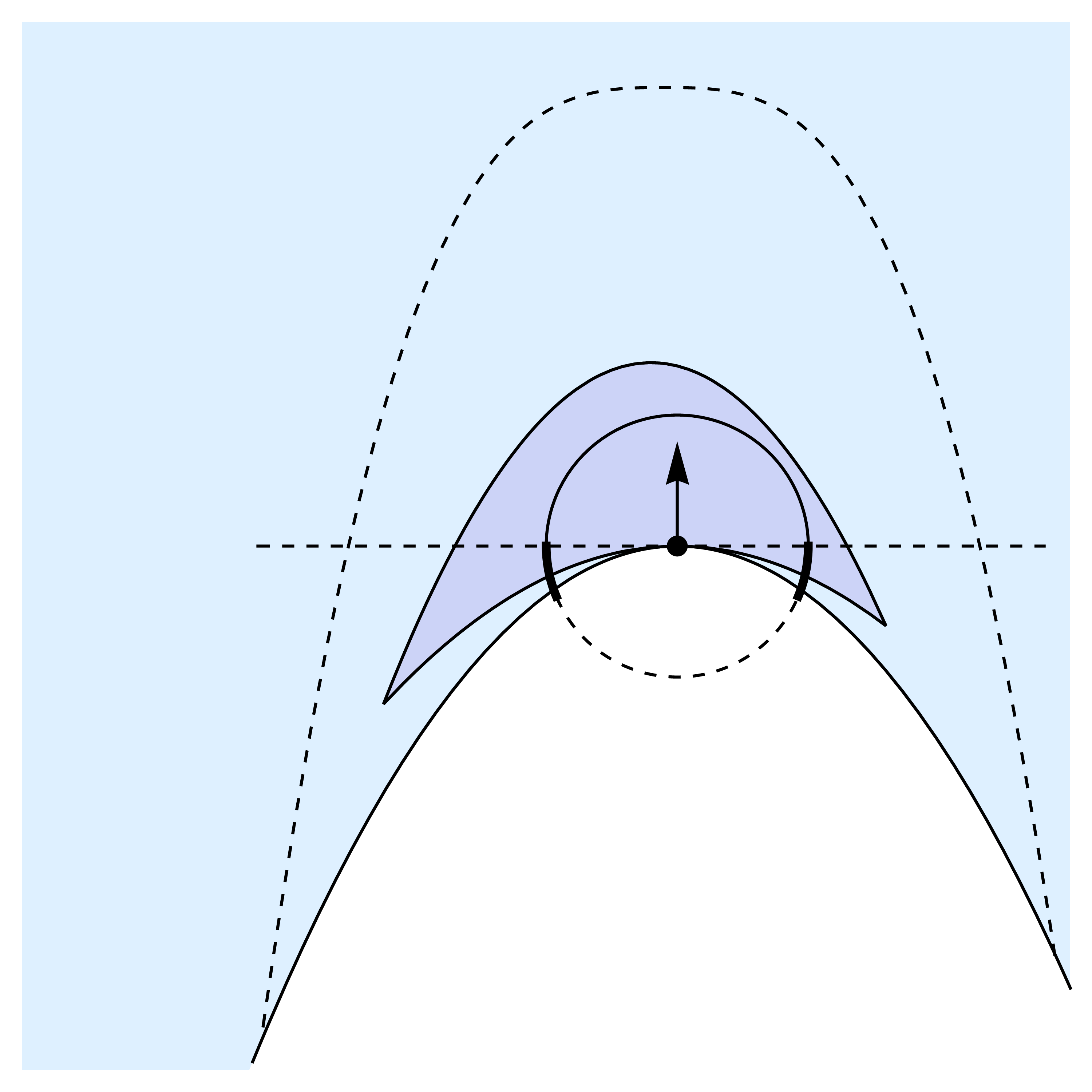}}
 \put(367,81){$\scriptstyle O$}
 \put(378,100){$\scriptstyle C_1$}
 \put(363,101){$\scriptstyle \vec n$}
 \put(378,79){$\scriptstyle \mathbf{C_2}$}
 \put(272,88){$\scriptstyle x_d=0$}
 \put(360,60){$\scriptstyle B_\rho(O)$}
 \put(300,140){$\scriptstyle \widetilde\sQ$}
 \put(272,116){$\scriptstyle x_d>0$}
 \put(321,40){$\scriptstyle \mydirac_0\!\widetilde\sQ$}
 \put(272,60){$\scriptstyle x_d<0$}
 \put(360,140){$\scriptstyle \widetilde N$}
 \put(318,79){$\scriptstyle \widetilde D^-$}
 \put(330,102){$\scriptstyle \widetilde D^+$}
 \end{picture}
 \caption[Quarter-ball in $(t,x)$-space, $\RR^{d+1}$, and its deformation.]{Quarter-ball in $(t,x)$-space, $\RR^{d+1}$, and its deformation for the case $\bar P\in\mydirac_0\!\sQ$.}
 \label{fig:parabolic_domain_interior_ball_and_deformation}
\end{figure}

\begin{prop}[Hopf-type lemma]
\label{prop:Friedman_lemma_2-2}
Let $\sQ\subset\RR^{d+1}$ be an open subset and assume that the coefficients of $L$ obey \eqref{eq:a_locally_strictly_parabolic_interior_domain}\footnote{It is sufficient for the proof of Proposition \ref{prop:Friedman_lemma_2-2} that $b^\perp(\bar t, \bar x) > 0$ in the case $(\bar t, \bar x) \in \mydirac_0\!\sQ$.} \eqref{eq:b_perp_positive_boundary_parabolic}, \eqref{eq:c_nonnegative_domain}, \eqref{eq:bc_locally_bounded}, \eqref{eq:a_continuous_degenerate_boundary}, and \eqref{eq:b_continuous_degenerate_boundary}. Assume that $u \in C^2_s(\underline\sQ)$
obeys $Lu\leq 0$ on $\sQ$ and that $u$ achieves a positive maximum $M$ in $\underline\sQ$. Suppose that $\underline\sQ$ contains the closure $\bar E$ of an open solid ellipsoid,
$$
E := \left\{(t,x) \in \RR^{d+1}: \gamma_0(t-t^*)^2+\sum_{i=1}^d\gamma_i(x_i-x_i^*)^2 < R^2\right\},
$$
where $\gamma_i>0$ for $i=0,1,\ldots,d$, and $R>0$, and that $u<M$ on $E$ and $u(\bar t, \bar x) = M$ at some point $(\bar t, \bar x) \in \partial E$. Then $\bar x = x^*$.
\end{prop}

\begin{proof}
The proof is the same as that of \cite[Lemma 2.2]{FriedmanPDE} when $\bar P:=(\bar t, \bar x) \in\sQ$. When $\bar P \in \mydirac_0\!\sQ$,  we shall adapt the method of proof of the Hopf boundary point lemma \cite[Lemma 4.1]{Feehan_maximumprinciple_v1}. We may assume that $\bar P$ is the only point in $\bar E$ where $u=M$ since, otherwise, we may confine our attention to a smaller ellipsoid lying within $E$ and having $\bar P$ as the only common point with $\partial E$. Thus,
\begin{equation}
\label{eq:u_strictly_less_M_punctured_ellipse}
u < M \quad\hbox{on }\bar E \less \{\bar P\},
\end{equation}
and $E$ is tangent to $\mydirac_0\!\sQ$ at $\bar P$. We now proceed by analogy with the proof of \cite[Lemma 4.1]{Feehan_maximumprinciple_v1}, denoting $P^*:=(t^*,x^*)$.

\setcounter{step}{0}
\begin{step}[Geometric simplification]
\label{step:Hopf_geometric_simplification}
We may assume without loss of generality, using a translation of $\RR^{d+1}$ if needed, that $\bar P=O\in\RR^{d+1}$. Moreover, using a diffeomorphism of $\RR^{d+1}$ defined by $(t,x) \mapsto (\tau,y)$, where
$$
\tau = \frac{\sqrt{\gamma_0}t}{R}, \quad y_i = \frac{\sqrt{\gamma_i}x_i}{R}, \quad i=1,\ldots,d,
$$
and then relabeling the coordinates $(\tau,y)$ as $(t,x)$ and relabeling the image $\widetilde P^* = (\tau^*,y^*)$ of $(t^*,x^*)$ again as $P^*=(t^*,x^*)$, we may assume that $E\subset\sQ$ is a unit ball,
$$
E = \left\{(t,x) \in \RR^{d+1}: (t-t^*)^2 + |x-x^*|^2 < 1\right\}.
$$
If $u(t,x) = v(\tau,y)$ and $\widetilde L$ is defined by setting $\widetilde Lv := Lu$, then $u_t = (\sqrt{\gamma_0}/R)v_\tau$ and we may divide the inequality $\widetilde Lv\leq 0$ by the positive constant $\sqrt{\gamma_0}/R$, so the coefficient of $v_\tau$ becomes $1$. We then relabel $v$ as $u$ and $(R/\sqrt{\gamma_0})\widetilde L$ as $L$.

The inward-pointing normal vector $n_0(O)e_0+\vec n(O)=\vec n(O)$ now lies along the line joining $O \in \bar E\cap\partial\sQ$ to the center $P^*$ of the unit ball, $E$. See the illustration on the left in Figure \ref{fig:parabolic_domain_interior_ball_and_deformation}. Because $n_0(O)=0$ by definition \eqref{eq:Degeneracy_locus_parabolic} of $\mydirac_0\!\sQ$, we must therefore have
$$
t^* = 0.
$$
If $x^*=0$, we are done (recall that $\bar x=0$ as a result of our initial translation), so to obtain a contradiction we suppose $x^*\neq 0$ and, with the aid of a rotation of $\RR^d$ (which necessarily fixes the origin), we may further suppose without loss of generality that
$$
x^* = (0,x_d^*)\in\RR^{d-1}\times\RR_+,
$$
that is, $x_d^*$ is \emph{positive} (as a direct consequence of our assumption that $\bar x\neq x^*$) and $x^*$ belongs to the open upper half-space, $\{x\in\RR^d:x_d>0\}$. Since $O\in\partial E$, we may set $(t,x)=(0,0)\in\partial E$ to give $(x_d^*)^2 = 1$, and so
$$
x_d^*=1.
$$
This completes the geometric simplification.
\end{step}

\begin{step}[Pushing downward using a $C^2$ diffeomorphism]
\label{step:Hopf_diffeomorphism}
Writing $x=(x',x_d)\in\RR^d$, we now choose a diffeomorphism,
\begin{equation}
\label{eq:Diffeomorphism}
\Phi:\RR^{d+1}\to\RR^{d+1}, \quad (t,x',x_d) \mapsto (t,x',y_d(t,x',x_d)),
\end{equation}
such that
\begin{align}
\label{eq:Diffeomorphism_fixes_origin}
{}&y_d(O)=0,
\\
\label{eq:Diffeomorphism_pushed_boundary_lower_halfspace}
{}&\Phi\left(\left\{(t,x) \in \RR^{d+1}: x_d < 3/4\right\} \cap \partial E\less\{O\}\right)\subset \left\{(t,x)\in\RR^{d+1}:x_d< 0\right\},
\\
{}&\Phi = \hbox{Identity map outside a bounded open neighborhood of } \bar E \subset \RR^{d+1}.
\end{align}
Because of \eqref{eq:Diffeomorphism_fixes_origin}, the diffeomorphism fixes the origin in $\RR^{d+1}$, and because of \eqref{eq:Diffeomorphism_pushed_boundary_lower_halfspace}, it pushes (parallel to the $x_d$-axis in $\RR^{d+1}$) the open portion of the sphere, $\{(t,x) \in \RR^{d+1}:x_d<3/4\}\cap\partial E\less\{O\}$, down below the hyperplane, $\{(t,x) \in \RR^{d+1}:x_d=0\}$. See the illustration on the right in Figure \ref{fig:parabolic_domain_interior_ball_and_deformation}.

Referring to the illustration on the left in Figure \ref{fig:parabolic_domain_interior_ball_and_deformation}, set
$$
D := \left\{(t,x) \in \RR^{d+1}: t^2 + |x'|^2 + x_d^2 \leq 1 \hbox{ and } 0\leq x_d \leq 3/4\right\} \subset\sQ.
$$
Henceforth, after applying the preceding diffeomorphism, and denoting $\Phi(D)=\widetilde D \subset \widetilde\sQ$, we may further assume, without loss of generality, that
$$
\widetilde D^+ := \widetilde D \cap \left\{(t,x)\in\RR^{d+1}: x_d>0\right\},
$$
has the property
$$
\left\{(t,x)\in\RR^{d+1}:x_d > 0\right\}\cap\partial\widetilde D^+ \Subset sQ.
$$
We also set
$$
\widetilde D^- := \widetilde D \cap \left\{(t,x)\in\RR^{d+1}: x_d<0\right\}.
$$
For convenience, we shall relabel $\widetilde\sQ=\Phi(\sQ) $ as $\sQ$ when this causes no confusion.
\end{step}

\begin{step}[Properties of the derivatives of the diffeomorphism]
\label{step:Diffeomorphism_derivative_properties}
We shall need to examine the effect of the diffeomorphism, $\Phi$, on the coefficients of $L$ and that it can be chosen to preserve the property \eqref{eq:b_perp_positive_boundary_parabolic} of the vector field, $b$, on a small enough neighborhood in $\mydirac_0\!\sQ$ of the point $O \in \mydirac_0\!\sQ$. Since our argument is be purely local, it suffices to define $\Phi$ on a neighborhood of the origin in $\RR^{d+1}$.

\begin{claim}
\label{claim:Diffeomorphism_derivative_properties}
The diffeomorphism $\Phi$ in \eqref{eq:Diffeomorphism} may be chosen so that its differential obeys
$$
\frac{\partial y_d}{\partial t}(O) = 0,
\quad
\frac{\partial y_d}{\partial x_i}(O) = 0 \quad\hbox{for } 1\leq i \leq d-1,
\quad\hbox{and}\quad
\frac{\partial y_d}{\partial x_d}(O) = 1.
$$
\end{claim}

\begin{proof}
We write
$$
y_d(t,x',x_d) = x_d - G(t,x') - x_d^2, \quad\forall\, (t,x',x_d) \in B_1(O)\subset \RR^{d+1},
$$
where $B_1(O)$ is the open unit ball with center at the origin, and
\begin{align*}
G(t,x') &:= 1 - \sqrt{1 - t^2 - |x'|^2},
\\
&\qquad\hbox{for $(t',x')\in \RR^d$ obeying } t^2 + |x'|^2 < 1,
\end{align*}
so that $G(t,x')=0$ when $x_d = G(t,x')$ and $(t,x) \in \{(t,x)\in\RR^{d+1}:0\leq x_d\leq 1\}\cap\partial E$. We compute that
\begin{gather*}
\frac{\partial y_d}{\partial t} = -\frac{\partial G}{\partial t},
\quad
\frac{\partial y_d}{\partial x_i} = -\frac{\partial G}{\partial x_i} \quad\hbox{for } 1\leq i \leq d-1,
\quad
\frac{\partial y_d}{\partial x_d} = 1 - 2x_d,
\\
\frac{\partial^2 y_d}{\partial x_i\partial x_j} = -\frac{\partial^2 G}{\partial x_i\partial x_j}
\quad\hbox{for } 1\leq i,j \leq d-1,\quad
\frac{\partial^2 y_d}{\partial x_i\partial x_d} = 0 \quad\hbox{for } 1\leq i \leq d-1,
\quad
\frac{\partial^2 y_d}{\partial x_d^2} = -2,
\end{gather*}
where
\begin{align*}
\frac{\partial G}{\partial t} &= \frac{t}{\left(1 - t^2 - |x'|^2\right)^{1/2}},
\\
\frac{\partial G}{\partial x_d} &= 0,
\\
\frac{\partial G}{\partial x_i} &= \frac{x_i}{\left(1 - t^2 - |x'|^2\right)^{1/2}},
\quad\hbox{for } 1\leq i\leq d-1,
\\
\frac{\partial^2 G}{\partial x_i\partial x_j} &= \frac{\delta_{ij}}{\left(1 - t^2 - |x'|^2\right)^{1/2}}
+ \frac{x_ix_j}{\left(1 - t^2 - |x'|^2\right)^{3/2}},
\\
&\qquad\hbox{for } 1\leq i,j \leq d-1,
\\
\frac{\partial^2 G}{\partial x_i\partial x_d} &= 0, \quad\hbox{for } 1\leq i\leq d-1.
\end{align*}
The properties of the derivatives of $y_d$ now follow by inspection.
\end{proof}

Since $a(O)=0$ by definition \eqref{eq:Degeneracy_locus_parabolic} of $\mydirac_0\!\sQ$ and the fact that $O\in \mydirac_0\!\sQ$ and our hypothesis \eqref{eq:a_continuous_degenerate_boundary} that $a\in C(\mydirac_0\!\sQ;\sS^+(d))$, it remains true that $a\circ\Phi^{-1}(O)=0$ and, moreover, that $O$ has an open neighborhood in $\partial\widetilde\sQ$ on which $a\circ\Phi^{-1}=0$. Similarly, because $n_0(P)=0$ for all points $P\in\partial\sQ$ in an open neighborhood in $\partial\sQ$ of the origin $O$ (by definition \eqref{eq:Degeneracy_locus_parabolic} of $\mydirac_0\!\sQ$), it remains true\footnote{Locally, we have $Q=(T',T)\times\sO$ for some open neighborhood $\sO\subset\RR^d$ of the origin and $\mydirac_0\!\sQ=(T',T)\times\partial_0\sO$ with inward-pointing normal vectors $n_0(P)e_0+\vec n(P)=\vec n(P)$ for all $P\in\mydirac_0\!\sQ$. Then $\widetilde\sQ=(T',T)\times\widetilde\sO$ and again all points $P\in\mydirac_0\!\widetilde\sQ=(T',T)\times\partial_0\widetilde\sO$ have inward-pointing normal vectors $n_0(P)e_0+\vec n(P)=\vec n(P)$.}
that $n_0(P)=0$ for all points $P\in\partial\widetilde\sQ$ in an open neighborhood in $\partial\widetilde\sQ$ of the origin $O$.

Because $\mydirac_0\!\widetilde\sQ$ is tangent at the origin to the hyperplane $\{(t,x)\in\RR^{d+1}:x_d=0\}$, we have
\begin{equation}
\label{eq:Normal_vector_origin}
\vec n(O)=e_d.
\end{equation}
See Figure \ref{fig:parabolic_domain_interior_ball_and_deformation}.
\end{step}

\begin{step}[Impact of the diffeomorphism on the open condition \eqref{eq:b_perp_positive_boundary_parabolic}]
\label{step:Impact_diffeomorphism_open_condition}
Writing $u(t,x',x_d) = v(t,x',y_d)$ and using $y_d=y_d(t,x',x_d)$, we obtain
$$
u_t = v_t + v_{y_d}\frac{\partial y_d}{\partial t}, \quad u_{x_i} = v_{y_i} + v_{y_d}\frac{\partial y_d}{\partial x_i}, \quad\hbox{for } 1\leq i \leq d-1, \quad u_{x_d} = v_{y_d}\frac{\partial y_d}{\partial x_d},
$$
and
\begin{align*}
u_{x_ix_j} &= v_{y_iy_j} + v_{y_iy_d}\frac{\partial y_d}{\partial x_j} + v_{y_jy_d}\frac{\partial y_d}{\partial x_i} + v_{y_dy_d}\frac{\partial y_d}{\partial x_j}\frac{\partial y_d}{\partial x_i} + v_{y_d}\frac{\partial^2 y_d}{\partial x_i\partial x_j}, \quad\hbox{for } 1\leq i,j \leq d-1,
\\
u_{x_ix_d} &= v_{y_iy_d}\frac{\partial y_d}{\partial x_d} + v_{y_dy_d}\frac{\partial y_d}{\partial x_i}\frac{\partial y_d}{\partial x_d}
+ v_{y_d}\frac{\partial^2 y_d}{\partial x_i\partial x_d}, \quad\hbox{for } 1\leq i \leq d-1,
\\
u_{x_dx_d} &= v_{y_dy_d}\left(\frac{\partial y_d}{\partial x_d}\right)^2 + v_{y_d}\frac{\partial^2 y_d}{\partial x_d^2}.
\end{align*}
Substituting the preceding derivative formulae into the expression \eqref{eq:Generator_parabolic} for $Lu$ and writing $Lu=\widetilde Lv$, we obtain
$$
\widetilde Lv = -v_t -\tr(\tilde aD^2v) - \langle \tilde b, Dv\rangle + \tilde cv,
$$
where (suppressing the arguments $(t,x',y_d)$ on the left and $(t,x',x_d)$ on the right),
\begin{align*}
\tilde a^{ij} &= a^{ij}, \quad\hbox{for } 1\leq i, j \leq d-1,
\\
\tilde a^{id} &= a^{id} + \sum_{j=1}^{d-1}a^{ij}\frac{\partial y_d}{\partial x_j}, \quad\hbox{for } 1\leq i \leq d-1,
\\
\tilde a^{dd} &= a^{dd}\left(\frac{\partial y_d}{\partial x_d}\right)^2 +  \sum_{i,j=1}^{d-1}a^{ij}\frac{\partial y_d}{\partial x_j}\frac{\partial y_d}{\partial x_i},
\\
\tilde b^i &= b^i, \quad\hbox{for } 1\leq i \leq d-1,
\\
\tilde b^d &= b^d + \sum_{i,j=1}^{d-1}a^{ij}\frac{\partial^2 y_d}{\partial x_i\partial x_j}
+ a^{dd}\frac{\partial^2 y_d}{\partial x_d^2}
+ \sum_{i=1}^{d-1}b^i\frac{\partial y_d}{\partial x_i} + \frac{\partial y_d}{\partial t},
\\
\tilde c &= c.
\end{align*}
It is clear that the required properties of the coefficients of $L$, namely  \eqref{eq:a_locally_strictly_parabolic_interior_domain} (and symmetry of $a$), \eqref{eq:c_nonnegative_domain}, \eqref{eq:bc_locally_bounded}, \eqref{eq:a_continuous_degenerate_boundary}, \eqref{eq:b_continuous_degenerate_boundary} are preserved by $\Phi$ and, also, that $v = u\circ\Phi^{-1}$ belongs to $C^2_s(\underline\sQ)$ by Definition \ref{defn:Second-order_boundary_regularity}.

By Claim \ref{claim:Diffeomorphism_derivative_properties}, we have
$$
\frac{\partial y_d}{\partial t}(O) = 0 \quad\hbox{and}\quad \frac{\partial y_d}{\partial x_i}(O) = 0 \quad\hbox{for } 1\leq i \leq d-1,
$$
and because $a^{ij}(O)=0$ by \eqref{eq:Degeneracy_locus_parabolic} (since $O\in\mydirac_0\!\sQ$), we obtain
$$
\sum_{i,j=1}^{d-1}a^{ij}(O)\frac{\partial^2 y_d}{\partial x_i\partial x_j}(O)
+ a^{dd}(O)\frac{\partial^2 y_d}{\partial x_d^2}(O)
+ \sum_{i=1}^{d-1}b^i(O)\frac{\partial y_d}{\partial x_i}(O) + \frac{\partial y_d}{\partial t}(O) = 0,
$$
and because $b^d(O)>0$ by \eqref{eq:b_perp_positive_boundary_parabolic} and \eqref{eq:Normal_vector_origin}, we must have
\begin{equation}
\label{eq:Positive_bd_boundarypoint}
\tilde b^d(O) > 0.
\end{equation}
We relabel $u$ and the coefficients of $L$ in the statement of Proposition \ref{prop:Friedman_lemma_2-2} as $u$ and $a,b,c$, respectively, after applying the diffeomorphism\footnote{Compare the proofs of \cite[Theorem 6.3.4]{Evans}, \cite[Lemma 6.5 or Theorem 8.12]{GilbargTrudinger}, \cite[Lemma 6.2.1]{Krylov_LecturesHolder} for similar arguments.}, $\Phi$.
\end{step}

\begin{step}[The barrier function and its properties]
\label{step:Barrier_when _positive_bc_function_holds}
We choose
\begin{equation}
\label{eq:Hopf_barrier_function}
h(t,x) := x_d, \quad (t,x) \in\RR^{d+1}.
\end{equation}
Clearly,
\begin{equation}
\label{eq:Barrier_function_signs}
h(t,x) \begin{cases} > 0 &\hbox{if } x_d>0, \\ = 0 &\hbox{if } x_d=0, \\ < 0 &\hbox{if } x_d<0. \end{cases}
\end{equation}
We observe that
$$
Lh = -b^dh_{x_d} + ch = -b^d + cx_d \quad\hbox{on }\sQ .
$$
Since $c$ is locally bounded on $\underline\sQ$ by \eqref{eq:c_locally_bounded}, and $b^d(O)>0$ by \eqref{eq:Positive_bd_boundarypoint} (and as $\tilde b^d$ was relabeled $b^d$), and $b^d$ is continuous at $O\in\underline\sQ$ by \eqref{eq:b_continuous_degenerate_boundary} (and the definition of $\tilde b^d$ in Step \ref{step:Impact_diffeomorphism_open_condition}), we may suppose that
\begin{equation}
\label{eq:Lbarrierfunction_negative}
Lh < 0 \quad\hbox{on } B_\rho(O),
\end{equation}
for an open ball $B_\rho(O)$ centered at the origin and small enough radius $\rho>0$.
\end{step}

\begin{step}[Application of the barrier function and the special case of the strong maximum principle]
\label{step:Barrier_contradiction}
We now argue as in the remainder of the proof of \cite[Lemma 2.2]{FriedmanPDE}. Recalling that $\bar P=(\bar t,\bar x)=O\in\RR^{d+1}$ and $x^*=(0,x_d^*)\in\{x\in\RR^d:x_d>0\}$, as a result of the application of our diffeomorphism of $\RR^{d+1}$ (in fact, $x_d^*=1$), choose $\rho>0$ small enough that
\begin{align*}
B_\rho(O)\cap \underline\sQ &\Subset \widetilde N,
\\
B_\rho(O)\cap \{(t,x)\in\RR^{d+1}:x_d>0\} &\Subset \widetilde D\cap \{(t,x)\in\RR^{d+1}:x_d\geq 0\}.
\end{align*}
See the illustration on the right in Figure \ref{fig:parabolic_domain_interior_ball_and_deformation}. Note that
$$
\mydirac_0(B_\rho(O)\cap\sQ ) = B_\rho(O)\cap\mydirac_0\!\sQ,
$$
and recall that, in our version of the strong maximum principle, we emphasize that points in $\mydirac_0\!\sQ$ behave in the same way as points in the \emph{interior}, $\sQ$. (By our assumption in this case, $\mydirac_0(B_\rho(O)\cap\sQ )$ contains the point $\bar P=O$ and thus is non-empty.) Write
$$
\partial(B_\rho(O)\cap\sQ) - \mydirac_0(B_\rho(O)\cap\sQ) = \bar C_1\cup C_2,
$$
where
$$
C_1 := \widetilde D^+ \cap \partial B_\rho(O),
$$
and $C_2$ is the complement of $\bar C_1$. See the illustration on the right in Figure \ref{fig:parabolic_domain_interior_ball_and_deformation}.

Clearly, \eqref{eq:u_strictly_less_M_punctured_ellipse} implies (after applying the diffeomorphism, $\Phi$) that
\begin{equation}
\label{eq:u_strictly_less_M_punctured_quarter_disk_image}
u < M \quad\hbox{on } \widetilde D\less\{O\},
\end{equation}
and because $\bar C_1 \subset \widetilde D\less\{O\}$, we have
\begin{equation}
\label{eq:u_less_M_C1}
u < M - \delta \quad\hbox{on } \bar C_1,
\end{equation}
for some constant $\delta>0$. Consider the function,
\begin{equation}
\label{eq:Defn_v}
v := u + \eps h \quad\hbox{on } \sQ ,
\end{equation}
for a constant $\eps>0$ to be chosen and observe that, because $Lu\leq 0$ on $\sQ$ (by hypothesis) and $Lh<0$ on $B_\rho(O)$ (by \eqref{eq:Lbarrierfunction_negative}),
\begin{equation}
\label{eq:Lv_negative_ball_intersect_Q}
Lv = Lu + \eps Lh < 0 \quad\hbox{on } B_\rho(O)\cap\underline\sQ.
\end{equation}
If $\eps>0$ is sufficiently small, then \eqref{eq:u_less_M_C1} and the fact that $h$ is continuous on $\RR^{d+1}$ imply that
$$
v < M \quad\hbox{on } \bar C_1.
$$
We also have $u\leq M$ on $C_2$ (since $u\leq M$ on $\underline\sQ$) and $h<0$ on $C_2$ (since $C_2\subset \{(t,x)\in\RR^{d+1}:x_d<0\}$ by \eqref{eq:Diffeomorphism_pushed_boundary_lower_halfspace} and $h<0$ on $\{(t,x)\in\RR^{d+1}:x_d<0\}$ by \eqref{eq:Barrier_function_signs}). Hence, we find that $v<M$ on $C_2$. Thus,
$$
v<M \quad\hbox{on } \bar C_1\cup C_2 = \partial(B_\rho(O)\cap\sQ) - \mydirac_0(B_\rho(O)\cap\sQ ),
$$
and $v(O)=u(O)=M$. Consequently, the function $v$ assumes a positive maximum in $B_\rho(O)\cap \underline\sQ$ while $Lv <0$ on $B_\rho(O)\cap\underline\sQ$, and this contradicts Lemma \ref{lem:Friedman_lemma_2-1}.
\end{step}
Therefore, $\bar x = x^*$ and this completes the proof of Proposition \ref{prop:Friedman_lemma_2-2}.
\end{proof}

Although we shall not use it in our proof of the strong maximum principle, a short addition to the proof of Proposition \ref{prop:Friedman_lemma_2-2} yields a boundary-degenerate parabolic analogue of the classical Hopf boundary point lemma for parabolic operators \cite[Theorem 2]{Friedman_1958}, \cite[Theorem 2 or $2'$]{Hill_1971}, \cite[Theorem 1]{Kusano_1963}, \cite[Lemma 2.8]{Lieberman}, \cite[Theorem 3.4]{Nazarov_2012}, and \cite[\S 3]{Wang_1990}. We first recall the statements of the classical weak maximum principle for $L$-subharmonic functions in $C^2(\sQ)$ or $W^{2,d+1}_{\loc}(\sQ)$, where a Dirichlet boundary condition is imposed along the full parabolic boundary, $\mydirac\!\sQ$.

\begin{thm}[Classical weak maximum principle for $L$-subharmonic functions in $C^2(\sQ)$]
\label{thm:Classical_weak_maximum_principle_C2_parabolic_relaxed}
\cite[Lemma 2.3]{Lieberman}
Let $\sQ\subset\RR^{d+1}$ be a bounded, open subset and $L$ in \eqref{eq:Generator_parabolic} with coefficients obeying \eqref{eq:a_nonnegative} and \eqref{eq:c_bounded_below}. Suppose $u\in C^2(\sQ)$ and $\sup_\sQ u < \infty$. If $Lu \leq 0$ on $\sQ$ and $u_*\leq 0$ on $\mydirac\!\sQ$, then $u\leq 0$ on $\sQ$.
\end{thm}

\begin{thm}[Classical weak maximum principle for $L$-subharmonic functions in $W^{2,d+1}_{\loc}(\sQ)$]
\label{thm:Classical_weak_maximum_principle_W2d+1_parabolic_relaxed}
Assume the hypotheses of Theorem \ref{thm:Classical_weak_maximum_principle_C2_parabolic_relaxed} on $\sQ$ and $L$, except that the coefficients of $L$ are now required to be measurable. Suppose $u\in W^{2,d+1}_{\loc}(\sQ)$ and $\sup_\sQ u < \infty$. If $Lu \leq 0$ a.e. on $\sQ$ and $u_*\leq 0$ on $\mydirac\!\sQ$, then $u\leq 0$ on $\sQ$.
\end{thm}

\begin{proof}
This follows from the classical weak maximum principle \cite[Corollary 7.4]{Lieberman} for a full Dirichlet boundary condition along $\mydirac\!\sQ$ and $L$-subharmonic functions in $W^{2,d+1}_{\loc}(\sQ)$ and our a priori weak maximum principle estimates,
Proposition \ref{prop:Parabolic_weak_maximum_principle_apriori_estimates}, using the method of proof of \cite[Theorem 2.18]{Feehan_perturbationlocalmaxima}, the elliptic analogue of Theorem \ref{thm:Classical_weak_maximum_principle_C2_parabolic_relaxed}.
\end{proof}

\begin{lem}[Hopf boundary point lemma for a degenerate-parabolic linear second-order differential operator]
\label{lem:Degenerate_hopf_lemma_parabolic}
Let $\sQ\subset\RR^{d+1}$ be an open subset and assume that the coefficients of $L$ in \eqref{eq:Generator_parabolic} obey \eqref{eq:a_locally_strictly_parabolic_interior_domain}, \eqref{eq:b_perp_positive_boundary_parabolic}\footnote{It is enough for the proof of Lemma \ref{lem:Degenerate_hopf_lemma_parabolic} that $b^\perp(\bar P) > 0$ in the case $\bar P \in \mydirac_0\!\sQ$.}, \eqref{eq:c_nonnegative_domain}, \eqref{eq:bc_locally_bounded}, \eqref{eq:a_continuous_degenerate_boundary}, and \eqref{eq:b_continuous_degenerate_boundary}. Suppose that $\underline\sQ$ contains $\bar B$, the closure  of an open ball,
$$
B := \left\{(t,x)\in\RR^{d+1}: |x-x^*|^2 + |t-t^*|^2 < R^2\right\} \subset\sQ,
$$
and $\bar P := (\bar t,\bar x)\in\partial B$ with $\bar x\neq x^*$. Suppose that $u\in C^2(\sQ) $ or $u\in W^{2,d+1}_{\loc}(\sQ) $ and that $u$ obeys
$$
Lu\leq 0 \quad\hbox{(a.e.) on } \sQ ,
$$
and that $u$ satisfies the conditions,
\begin{enumerate}
\renewcommand{\theenumi}{\roman{enumi}}
\item $u$ is continuous at $\bar P$;
\item\label{item:xzero_strict_local_max_parabolic} $u(\bar P) > u(P)$, for all $P\in B$;
\item $D_{\vec n} u(\bar P)$ exists,
\end{enumerate}
where $D_{\vec n} u(\bar P)$ is the derivative of $u$ at $\bar P$ in the direction of the \emph{inward}-pointing unit normal vector, $\vec n(\bar P)$, at $(\bar P)\in\partial B$. Then the following hold:
\begin{enumerate}
\item\label{item:Hopf_c_zero_parabolic} If $c=0$ on $\sQ$, then $D_n u(\bar P)$ obeys the strict inequality,
\begin{equation}
\label{eq:Positive_inward_normal_derivative_parabolic}
D_{\vec n} u(\bar P) < 0.
\end{equation}
\item\label{item:Hopf_c_geq_zero_parabolic} If $c\geq 0$ on $\sQ$ and $u(\bar P)\geq 0$, then \eqref{eq:Positive_inward_normal_derivative_parabolic} holds.
\item\label{item:Hopf_c_no_sign_parabolic} If $u(\bar P)=0$, then \eqref{eq:Positive_inward_normal_derivative_parabolic} holds irrespective of the sign of $c$.
\end{enumerate}
\end{lem}

\begin{rmk}[On the hypothesis of strict interior local parabolicity]
\label{rmk:Hopf_lemma_strict_interior_local_parabolicity}
When $\bar P \in \mydirac_0\!\sQ$, the hypothesis that $a$ obeys \eqref{eq:a_locally_strictly_parabolic_interior_domain} can be omitted.
\end{rmk}

\begin{rmk}[Differences between the regularity hypotheses on $u$ in Proposition \ref{prop:Friedman_lemma_2-2} and Lemma \ref{lem:Degenerate_hopf_lemma_parabolic}]
\label{rmk:Hypotheses_hopf_and_hopf_type_lemmas}
While the hypotheses of the `Hopf-type lemma', Proposition \ref{prop:Friedman_lemma_2-2}, require that $u\in C^2_s(\underline\sQ)$, the hypotheses of the Hopf boundary point Lemma \ref{lem:Degenerate_hopf_lemma_parabolic} only require that $u\in C^2(\sQ)$ (respectively, $W^{2,d+1}_{\loc}(\sQ)$ when $L$ has measurable coefficients), $u$ is continuous at the boundary point, $\bar P$, and $D_{\vec n} u(\bar P)$ exists. Note also that while Lemma \ref{lem:Degenerate_hopf_lemma_parabolic} allows $u\in C^2(\sQ)$ or $W^{2,d+1}_{\loc}(\sQ)$, that is not true for Proposition \ref{prop:Friedman_lemma_2-2}, which requires that $u\in C^2_s(\underline\sQ)$.
\end{rmk}

\begin{proof}[Proof of Lemma \ref{lem:Degenerate_hopf_lemma_parabolic}]
When $\bar B\subset\sQ$, the conclusion \eqref{eq:Positive_inward_normal_derivative_parabolic} follows from the classical Hopf boundary point lemma for a parabolic linear second-order differential operator \cite[Theorem 2]{Friedman_1958}, \cite[Theorem 1]{Kusano_1963}, so it suffices to consider the case where $\bar B\cap\mydirac_0\!\sQ = \{\bar P\}$. For this purpose, we continue the notation and geometric setup employed in the proof of Proposition \ref{prop:Friedman_lemma_2-2}. We need only supplement the arguments in Step \ref{step:Barrier_contradiction} of the proof of Proposition \ref{prop:Friedman_lemma_2-2} to obtain the conclusion.

\setcounter{step}{0}
\begin{step}[Verification that the classical weak maximum principle holds for $L$ on $B^+_\rho(O)$]
\label{step:Degenerate_hopf_lemma_proof_weak_max_principle}
Theorem \ref{thm:Classical_weak_maximum_principle_C2_parabolic_relaxed} implies that the classical weak maximum principle (that is, with full boundary comparison) holds for the operator $L$ on $B^+_\rho(O) = \{x_d>0\}\cap B_\rho(O)$ and $L$-subharmonic functions
$$
w \in C^2(B^+_\rho(O))\cap C(\bar B^+_\rho(O)).
$$
Theorem \ref{thm:Classical_weak_maximum_principle_W2d+1_parabolic_relaxed} implies that the classical weak maximum principle holds for $L$ on $B^+_\rho(O)$ and $L$-subharmonic functions
$$
w \in W^{2,d+1}_{\loc}(B^+_\rho(O))\cap C(\bar B^+_\rho(O)),
$$
concluding this step.
\end{step}

\begin{step}[Application of the classical weak maximum principle]
\label{step:Degenerate_hopf_lemma_proof_application_weak_max_prin}
From \eqref{eq:Lv_negative_ball_intersect_Q}, we obtain $Lv<0$ (a.e.) on $B^+_\rho(O)$, where we recall that $v = u + \eps h$ from \eqref{eq:Defn_v}. Since $u-u(O)<0$ on $\widetilde D\less\{O\}$ by \eqref{eq:u_strictly_less_M_punctured_quarter_disk_image} (see Figure \ref{fig:parabolic_domain_interior_ball_and_deformation}) and $u\in C(\sQ) $ and
$$
\{x_d>0\}\cap\partial B_\rho^+(O) = C_1 \Subset \widetilde D\less\{O\},
$$
we obtain, writing $P=(t,x)=(t,x',x_d)\in\RR^{d+1}$,
$$
u(P) - u(O) \leq -m < 0, \quad\forall\, P \in \{x_d>0\}\cap\partial B_\rho^+(O),
$$
for some positive constant, $m$, depending on $\rho$ and $u$. But
$$
h(P) = x_d \leq \rho, \quad\forall\, P \in \{x_d>0\}\cap\partial B_\rho^+(O).
$$
Consequently,
$$
u(P) - u(O) + \eps h(P) \leq -m + \eps \rho \leq 0, \quad \forall\, P \in \{x_d>0\}\cap\partial B_\rho^+(O),
$$
provided we fix $\eps$ in the range $0<\eps\leq m/\rho$, while
$$
u(P) - u(O) + \eps h(P) = u(P) - u(O) \leq 0, \quad \forall\, P \in \{x_d=0\}\cap\partial B_\rho^+(O),
$$
since (trivially) $h(P)=0$ when $x_d=0$ and \eqref{eq:u_strictly_less_M_punctured_quarter_disk_image} implies that $u(P) \leq u(O)$ on $\partial B_\rho^+(O)\subset \widetilde D\cup\{O\}$. But
$$
L(u - u(O) + \eps h) = Lu - cu(O) +\eps Lh \leq -cu(O) \leq 0 \quad\hbox{on } B_\rho^+(O),
$$
where the last inequality holds if $c=0$ on $\sQ$ (as in Conclusion \eqref{item:Hopf_c_zero_parabolic}), or $c\geq 0$ on $\sQ$ and $u(O)\geq 0$  (as in Conclusion \eqref{item:Hopf_c_geq_zero_parabolic}), or $c$ has arbitrary sign on $\sQ$ and $u(O)=0$ (as in Conclusion \eqref{item:Hopf_c_no_sign_parabolic}). (For the case $u(O)=0$, we simply note as in the proof of \cite[Lemma 3.4]{GilbargTrudinger} that we can replace $L$ by $L+c^-$, where we write $c=c^+-c^-$.)

By Step \ref{step:Degenerate_hopf_lemma_proof_weak_max_principle}, for $u\in C^2(\sQ) $ or $W^{2,d+1}_{\loc}(\sQ) $, we can apply the classical weak maximum principle to $v = u + \eps h$ on $B^+_\rho(O)$, yielding
\begin{equation}
\label{eq:Hopf_lemma_key_half_ball_inequality}
u - u(O) + \eps h \leq 0 \quad\hbox{on } B^+_\rho(O),
\end{equation}
since $L(u - u(O) + \eps h) \leq 0$ on $B^+_\rho(O)$ and $u - u(O) + \eps h \leq 0$ on $\partial B^+_\rho(O)$.
\end{step}

\begin{step}[Sign of the directional derivative of the subsolution at the boundary]
\label{step:DegenerateHopfLemmaProof_SignDirectionalDeriv}
From \eqref{eq:Hopf_barrier_function}  and \eqref{eq:Hopf_lemma_key_half_ball_inequality}, we have
$$
\frac{1}{x_d}\left(u(0,0,x_d)-u(O)\right) \leq -\frac{\eps}{x_d}h(0,0,x_d) = -\eps, \quad\forall\,(0,0,x_d)\in B^+_\rho(O).
$$
Taking the limit in the preceding inequality as $x_d\downarrow 0$ yields
$$
u_{x_d}(O) \leq -\eps < 0,
$$
and thus \eqref{eq:Positive_inward_normal_derivative_parabolic} holds.
\end{step}
This completes the proof.
\end{proof}

We have an analogue of \cite[Lemma 2.3]{FriedmanPDE}.

\begin{lem}
\label{lem:Friedman_lemma_2-3}
Let $\sQ\subset\RR^{d+1}$ be an open subset and assume that the coefficients of $L$ obey \eqref{eq:a_locally_strictly_parabolic_interior_domain}, \eqref{eq:b_perp_positive_boundary_parabolic}, \eqref{eq:c_nonnegative_domain}, \eqref{eq:bc_locally_bounded}, \eqref{eq:a_continuous_degenerate_boundary}, and \eqref{eq:b_continuous_degenerate_boundary}. If $u \in C^2_s(\underline\sQ)$
obeys $Lu\leq 0$ on $\sQ$ and $u$ has a positive maximum in $\underline\sQ$ which is attained at a point $P^0$, then $u(P) = u(P^0)$ for all points $P \in C(P^0)$.
\end{lem}

\begin{proof}
The proof is identical to that of \cite[Lemma 2.3]{FriedmanPDE}, except that the role of \cite[Lemma 2.2]{FriedmanPDE} is replaced by that of Proposition \ref{prop:Friedman_lemma_2-2}.
\end{proof}

We have an analogue of \cite[Lemma 2.4]{FriedmanPDE}; note that the interval for $t$ is \emph{forward} in time here, consistent with the convention in this article and \cite{Bensoussan_Lions} of considering a terminal value problem, rather than \emph{backward} in time as in \cite[Lemma 2.4]{FriedmanPDE}, consistent with Friedman's convention of considering an initial value problem.

\begin{lem}
\label{lem:Friedman_lemma_2-4}
Let $\sQ\subset\RR^{d+1}$ be an open subset and assume that the coefficients of $L$ obey \eqref{eq:a_locally_strictly_parabolic_interior_domain}, \eqref{eq:b_perp_positive_boundary_parabolic}, \eqref{eq:c_nonnegative_domain}, \eqref{eq:c_nonnegative_boundary}, \eqref{eq:bc_locally_bounded}, \eqref{eq:a_continuous_degenerate_boundary}, and \eqref{eq:b_continuous_degenerate_boundary}. Assume that $\underline\sQ$ contains the closure $\bar R$ of an \emph{open} rectangle,
$$
R := \left\{(t,x) \in\sQ: t^0 < t < t^0 + a_0, \ |x_i-x_i^0| < a_i \hbox{ for } i=1,\ldots,d\right\},
$$
where $a_i>0$ for $i=0,1,\ldots,d$. If $u \in C^2_s(\underline\sQ)$ obeys $Lu\leq 0$ on $\sQ$ and $u$ has a positive maximum in $\bar R$ which is attained at the point $P^0=(t^0,x^0) \in \underline R$, then $u(P) = u(P^0)$ for all points $P \in \bar R$.
\end{lem}

\begin{proof}
The proof is identical to that of \cite[Lemma 2.4]{FriedmanPDE}, except that the roles of \cite[Lemmas 2.1 and 2.3]{FriedmanPDE} are replaced by those of Lemmas \ref{lem:Friedman_lemma_2-1} and \ref{lem:Friedman_lemma_2-3}.
\end{proof}

Finally, we have the following analogue of \cite[Theorem 2.1]{FriedmanPDE}.

\begin{thm}[Strong maximum principle when $c\geq 0$]
\label{thm:Friedman_theorem_2-1}
Let $\sQ\subset\RR^{d+1}$ be an open subset and assume that the coefficients of $L$ obey \eqref{eq:a_locally_strictly_parabolic_interior_domain}, \eqref{eq:b_perp_positive_boundary_parabolic}, \eqref{eq:c_nonnegative_domain}, \eqref{eq:c_nonnegative_boundary}, \eqref{eq:bc_locally_bounded}, \eqref{eq:a_continuous_degenerate_boundary}, and \eqref{eq:b_continuous_degenerate_boundary}. If $u \in C^2_s(\underline\sQ)$ obeys $Lu\leq 0$ on $\sQ$ and $u$ has a global positive maximum which is attained at a point $P^0 \in \underline\sQ$, then $u = u(P^0)$ on $S(P^0)$.
\end{thm}

\begin{proof}
The proof is identical to that of \cite[Theorem 2.1]{FriedmanPDE}, except that the role of \cite[Lemma 2.4]{FriedmanPDE} is replaced by that of Lemma \ref{lem:Friedman_lemma_2-4}.
\end{proof}

\begin{rmk}[Strong maximum principle for parabolic operators with multiple time coordinates]
When $m=1$, and $b_{ij}=0$ on $\sQ$, and $b_1=1$ on $\sQ$ in the notation of \cite[Equation (2.1)]{FriedmanPDE}, then \cite[Theorem 2.2]{FriedmanPDE} reduces to \cite[Theorem 2.1]{FriedmanPDE}, albeit with the weaker conclusion, namely that $u = u(P^0)$ on $C(P^0)$ rather than $u = u(P^0)$ on $S(P^0)$. (Recall that $C(P^0) \subset S(P^0)$.)
\end{rmk}

We now relax the requirement that $c\geq 0$ on $\underline\sQ$ and give the following analogue of \cite[Theorem 2.3]{FriedmanPDE}.

\begin{thm}[Strong maximum principle when $c$ has arbitrary sign]
\label{thm:Friedman_theorem_2-3}
Let $\sQ\subset\RR^{d+1}$ be an open subset and assume that the coefficients of $L$ obey \eqref{eq:a_locally_strictly_parabolic_interior_domain}, \eqref{eq:b_perp_positive_boundary_parabolic}, \eqref{eq:c_nonnegative_domain}, \eqref{eq:c_nonnegative_boundary}, \eqref{eq:bc_locally_bounded}, and \eqref{eq:abc_continuous_degenerate_boundary}. Suppose $u \in C^2_s(\underline\sQ)$ obeys $Lu\leq 0$ on $\sQ$. If $u\leq 0$ on $\underline\sQ$ and $u(P^0)=0$ for some $P^0 \in \underline\sQ$, then $u = 0$ on $C(P^0)$.
\end{thm}

\begin{proof}
Let $B_\rho(P^1)\subset\RR^{d+1}$ be an open ball of radius $\rho>0$ centered at a point $P^1\in C(P^0)$. If $B_\rho(P^1)\Subset\sQ$, then the proof of \cite[Theorem 2.3]{FriedmanPDE} yields $u=0$ on $B_\rho(P^1)\cap C(P^0)$, so it suffices to consider the case where $B_\rho(P^1)$ is centered at a point $P^1\in\mydirac_0\!\sQ$. We may assume without loss of generality (by a translation of the spatial coordinates, $x_1,\ldots,x_d$) that $P^1=(t^1,0)$ and (by a rotation of the spatial coordinates, $x_1,\ldots,x_d$) that $\vec n(P^1) = e_d$.

Define $v := e^{-\sigma x_d}u$ on $\sQ$, for a positive constant $\sigma$ to be chosen later, and observe that our hypothesis on $u$ yields
$$
v \leq 0 \quad\hbox{on }\sQ.
$$
A calculation yields
\begin{align*}
e^{\sigma x_d}Lu &= -v_t - a^{ij}v_{x_ix_j} -\left(b^i+2\sigma a^{id}\right) v_{x_i} + \left(c - \sigma b^d - \sigma^2 a^{dd}\right)v
\\
&= -v_t - a^{ij}v_{x_ix_j} - \tilde b^iv_{x_i} + \left(c - \sigma b^d - \sigma^2 a^{dd}\right)v
\\
&=: L_0v + \left(c - \sigma b^d - \sigma^2 a^{dd}\right)v,
\end{align*}
where the coefficient of $v$ in $L_0v$ is zero. We again have $v \in C^2(\sQ) \cap C^1(\underline\sQ)$ with $\tr(aD^2v) \in C(\underline\sQ)$ and $\tr(aD^2v)=0$ on $\mydirac_0\!\sQ$, so $v\in C^2_s(\underline\sQ)$. Moreover, $Lu \leq 0$ on $\sQ$ implies that
$$
L_0v \leq -\left(c - \sigma b^d - \sigma^2 a^{dd}\right)v \quad\hbox{on }\sQ.
$$
The coefficients $\tilde b^i := b^i-2\sigma a^{id}$ are continuous along $\mydirac_0\!\sQ$ by \eqref{eq:abc_continuous_degenerate_boundary}. Since $a=0$ on $\mydirac_0\!\sQ$, we have $a(P^1)=0$ and so
$$
\tilde b^\perp(P^1) = \tilde b^d(P^1) = b^d(P^1) > 0,
$$
and so \eqref{eq:b_perp_positive_boundary_parabolic} holds for $\tilde b^\perp$ on $B_\rho(P^1)\cap\mydirac_0\!\sQ$ for small enough $\rho$. Moreover,
$$
\left(c - \sigma b^d - \sigma^2 a^{dd}\right)(P^1) = c(P^1) - \sigma b^d(P^1),
$$
and so, for a large enough constant $\sigma=\sigma(c(P^1),b^d(P^1))$, we obtain
$$
\left(c - \sigma b^d - \sigma^2 a^{dd}\right)(P^1) < 0.
$$
Because the coefficients $a^{dd}, b^d, c$ are continuous at $P^1 \in \mydirac_0\!\sQ$ by \eqref{eq:abc_continuous_degenerate_boundary}, for a small enough radius $\rho$, we have
$$
c - \sigma b^d - \sigma^2a^{dd}  \leq 0 \quad\hbox{on } B_\rho(P^1)\cap\sQ.
$$
Therefore, $L_0v \leq 0$ on $B_\rho(P^1)\cap\sQ$. Since the coefficient of $v$ in $L_0v$ is zero (in particular, nonnegative) on $B_\rho(P^1)\cap\sQ$, while $L_0(v+1)=L_0v\leq 0$ on $B_\rho(P^1)\cap\sQ$ and $v+1\leq (v+1)(P^0)=1$ on $B_\rho(P^1)\cap\sQ$, then Theorem \ref{thm:Friedman_theorem_2-1} applies to give $v+1=v(P^0)+1$ on $B_\rho(P^1)\cap C(P^0)$, and thus $u=u(P^0)=0$ on $B_\rho(P^1)\cap C(P^0)$. Therefore, the subset of points $P\in C(P^0)$ where $u(P)=0$ is open and, because this subset is necessarily closed (since $u$ is continuous on $\underline\sQ$) and $C(P^0)$ is connected, we must have $u=0$ on $C(P^0)$.
\end{proof}

The following refinement of Theorem \ref{thm:Friedman_theorem_2-1}, analogous to \cite[Theorem 2.4]{FriedmanPDE}, makes a stronger assertion since the hypotheses only assume that $u(P^0)$ is the maximum of $u$ on $S(P^0)\subset \underline\sQ$ rather than $\underline\sQ$.

\begin{thm}[Refined strong maximum principle when $c\geq 0$]
\label{thm:Friedman_theorem_2-4}
Let $\sQ\subset\RR^{d+1}$ be an open subset and assume that the coefficients of $L$ obey \eqref{eq:a_locally_strictly_parabolic_interior_domain}, \eqref{eq:b_perp_positive_boundary_parabolic}, \eqref{eq:c_nonnegative_domain}, \eqref{eq:c_nonnegative_boundary}, \eqref{eq:b_locally_bounded},\eqref{eq:c_locally_bounded}, \eqref{eq:a_continuous_degenerate_boundary}, and \eqref{eq:b_continuous_degenerate_boundary}. If $u \in C^2_s(\underline\sQ)$ obeys $Lu\leq 0$ on $S(P^0)$, and $c\geq 0$ on $S(P^0)$, and $u$ has a global positive maximum which is attained at the point $P^0$, then $u = u(P^0)$ on $S(P^0)$.
\end{thm}

\begin{proof}
The proof is the same as that of Theorem \ref{thm:Friedman_theorem_2-1}, since we only made use of the fact that $u(P^0)$ is the maximum of $u$ on $S(P^0)\subset \underline\sQ$ (and not necessarily the maximum on $\underline\sQ$).
\end{proof}

We have the following analogue of \cite[Theorem 2.5]{FriedmanPDE}.

\begin{thm}[Refined strong maximum principle when $c$ has arbitrary sign]
\label{thm:Friedman_theorem_2-5}
Let $\sQ\subset\RR^{d+1}$ be an open subset and assume that the coefficients of $L$ obey \eqref{eq:a_locally_strictly_parabolic_interior_domain}, \eqref{eq:b_perp_positive_boundary_parabolic}, \eqref{eq:c_nonnegative_domain}, \eqref{eq:c_nonnegative_boundary}, \eqref{eq:bc_locally_bounded}, and \eqref{eq:abc_continuous_degenerate_boundary}. If $u \in C^2_s(\underline\sQ)$ obeys $Lu\leq 0$ on $\sQ$, and $u\leq 0$ on $S(P^0)$, and $u(P^0)=0$, then $u = 0$ on $S(P^0)$.
\end{thm}

\begin{proof}
The proof is identical to that of \cite[Theorem 2.5]{FriedmanPDE}, except that the roles of \cite[Theorem 2.3]{FriedmanPDE} and its method of proof and the proof of \cite[Theorem 2.1]{FriedmanPDE} are replaced by those of Theorems \ref{thm:Friedman_theorem_2-3} and \ref{thm:Friedman_theorem_2-1}.
\end{proof}

As in \cite[\S 2.2]{FriedmanPDE}, we can deduce a version\footnote{Compare Theorem \ref{thm:Weak_maximum_principle_C2s_bounded_domain}.} of the weak maximum principle from the strong maximum principle and obtain the following analogue of \cite[Theorem 2.6]{FriedmanPDE}.

\begin{thm}[Weak maximum principle]
\label{thm:Friedman_theorem_2-6}
Let $\sQ\subset\RR^{d+1}$ be a \emph{bounded} open subset and assume that the coefficients of $L$ obey \eqref{eq:a_locally_strictly_parabolic_interior_domain}, \eqref{eq:b_perp_positive_boundary_parabolic}, \eqref{eq:c_nonnegative_domain}, \eqref{eq:c_nonnegative_boundary}, \eqref{eq:bc_locally_bounded}, \eqref{eq:a_continuous_degenerate_boundary}, and \eqref{eq:b_continuous_degenerate_boundary}. If $u \in C^2_s(\underline\sQ)$ obeys $Lu\leq 0$ on $\sQ$ and $u^*$ attains a global positive maximum at a point in $\bar S(P^0)$, then $u^*$ attains that maximum value at a point in the complement of $S(P^0)\cup \mydirac_0\!S(P^0)$.
\end{thm}

\begin{proof}
The proof is identical to that of \cite[Theorem 2.6]{FriedmanPDE}, except that the role of \cite[Theorem 2.4]{FriedmanPDE} is replaced by that of Theorem \ref{thm:Friedman_theorem_2-4}.
\end{proof}

We have the following analogue of the \cite[Remark, p. 40]{FriedmanPDE}.

\begin{rmk}[Maxima of arbitrary sign]
If in Theorems \ref{thm:Friedman_theorem_2-1}, \ref{thm:Friedman_theorem_2-4}, and \ref{thm:Friedman_theorem_2-6} we have $c=0$ on $\sQ$, then for any constant $k\in\RR$, we have $L(u+k) = Lu$. Consequently, all the assertions remain true if the maximum value of $u$ is not assumed to be positive.
\end{rmk}

\appendix

\section{Fichera weak maximum principle and the parabolic Heston operator}
\label{sec:Fichera_and_heston}
We can compare the weak maximum principles and uniqueness theorems provided by our Theorems \ref{thm:Weak_maximum_principle_C2s_bounded_domain} and \ref{thm:Weak_maximum_principle_C2s_unbounded_domain} with those of Fichera, Ole{\u\i}nik, and Radkevi{\v{c}} \cite{Radkevich_2009a} in the case of the parabolic Heston operator, $L$, in \S \ref{subsec:Heston} on $\sO_T = (0,T)\times\sO$, for an open subset $\sO\subseteqq\RR\times\RR_+$ and show that those of Fichera, Ole{\u\i}nik, and Radkevi{\v{c}} are strictly weaker when $0<\beta<1$, where we recall from \S \ref{subsec:Heston} that $\beta = 2\kappa\theta/\sigma^2$.

Following the exposition by Z. Wu, J. Yin, and C. Wang in \cite[pp. 357--358]{Wu_Yin_Wang_2006}, we shall regard $L$ as a degenerate-elliptic operator, denoting $t=x_0$, in order to apply the Fichera maximum principles and uniqueness results described in \cite[Chapter 1]{Radkevich_2009a}. In the framework of Fichera (see \cite[p. 308]{Radkevich_2009a}), we let\footnote{In the work of Fichera \cite{Fichera_1960, Oleinik_Radkevic, Radkevich_2009a, Radkevich_2009b}, the boundary of the open subset $\sO_T\subset\RR^{d+1}$ is usually denoted by $\Sigma$ and $\Sigma^0$ is the subset of points $(t,x)\in\Sigma$ where $a^{ij}(t,x)n_in_j=0$.} $\Sigma$ denote the subset of points $(t,x)\in\partial\sO_T$ where $a^{ij}(t,x)n_in_j=0$ (with $\vec n:\partial\sO_T\to\RR^{d+1}$ denoting the \emph{inward}-pointing unit normal vector field along $\partial\sO_T$, as in \cite[p. 308]{Radkevich_2009a}) and the \emph{Fichera function} \cite[Equations (1.1.2) and (1.1.3)]{Radkevich_2009a} (taking into account our sign convention in \eqref{eq:Generator_parabolic} for the coefficients $(a,b,c)$ of $L$) is
$$
\fb := \sum_{k=0}^d\left(b^k - a^{kj}_{x_j}\right)n_k = \sum_{k=1}^d\left(b^k - a^{kj}_{x_j}\right)n_k + n_0,
$$
noting that $b^0\equiv 1$ and $a^{ij}\equiv 0$ when $i=0$ or $j=0$. For the parabolic Heston operator, $L$, in \eqref{eq:Heston_generator}, we have
$$
\Sigma = (0,T)\times\partial_0\sO\ \cup \{0\}\times\sO \cup \{T\}\times\sO,
$$
since $a^{ij}(t,x_1,x_2)=0$ when $x_2=0$ and $\vec n = (1,0,0)$ or $(-1,0,0)$, respectively, when $t=0$ or $T$.

Following \cite[p. 308]{Radkevich_2009a}, we denote by $\Sigma_1\subset\Sigma$ the subset where $\fb > 0$, by $\Sigma_2\subset\Sigma$ the subset where $\fb < 0$, and by $\Sigma_0\subset\Sigma$ the subset where $\fb = 0$; the set $\partial\sO_T\less\Sigma$ is denoted by $\Sigma_3$. By \cite[Theorem 1.1.1]{Radkevich_2009a}, the characterization of the subsets $\Sigma, \Sigma_0, \Sigma_1, \Sigma_2, \Sigma_3$ of the boundary $\partial\sO_T$ remains invariant under smooth changes of the independent coordinates, $(x_0,x_1,\ldots,x_d)$.

For the Heston operator, $L$, when $(t,x_1,x_2) \in (0,T)\times\partial_0\sO$, we have
$$
\fb(t,x_1,x_2)
=
\begin{cases}
1, &\hbox{if } t=0 \hbox{ and } (x_1,x_2) \in \sO,
\\
-1, &\hbox{if } t=T \hbox{ and } (x_1,x_2) \in \sO,
\\
\sigma^2(\beta-1)/2, &\hbox{if } t\in (0,T) \hbox{ and } (x_1,0) \in \partial\sO.
\end{cases}
$$
Hence,
\begin{align*}
\Sigma_0 &= (0,T)\times\partial_0\sO, \quad\hbox{if } \beta = 1,
\\
\Sigma_1 &=
\begin{cases}
\{0\}\times\sO, &\hbox{if } 0<\beta\leq 1,
\\
\{0\}\times\sO \cup (0,T)\times\partial_0\sO, &\hbox{if } \beta > 1,
\end{cases}
\\
\Sigma_2 &=
\begin{cases}
\{T\}\times\sO \cup (0,T)\times\partial_0\sO, &\hbox{if } 0<\beta<1,
\\
\{T\}\times\sO, &\hbox{if } \beta\geq 1,
\end{cases}
\\
\Sigma_3 &= (0,T)\times\partial_1\sO, \quad\hbox{if } \beta>0.
\end{align*}
The \emph{first boundary value problem of Fichera} \cite[Equations (1.1.4) and (1.1.5)]{Radkevich_2009a} for the operator $L$ is to find a function $u\in C^2(\sO_T)\cap C(\bar\sO_T)$ such that
$$
Lu = f\hbox{ on } \sO_T, \quad u = g \hbox{ on } \Sigma_2\cup\Sigma_3,
$$
given a source function $f$ on $\sO_T$ and a boundary data function $g$ on $\Sigma_2\cup\Sigma_3$. But
$$
\Sigma_2\cup\Sigma_3
=
\begin{cases}
\{T\}\times\sO \cup (0,T)\times\partial\sO &\hbox{if }0<\beta<1,
\\
\{T\}\times\sO \cup (0,T)\times\partial_1\sO &\hbox{if }\beta \geq 1.
\end{cases}
$$
Thus, for the parabolic Heston operator and $g\in C(\bar\Sigma_2\cup\bar\Sigma_3)$, the first boundary value problem of Fichera becomes
$$
Lu = f\hbox{ on } \sO_T, \quad u = g \hbox{ on }
\begin{cases}
\mydirac\!\sO_T &\hbox{if }0<\beta<1,
\\
\mydirac_1\!\sO_T &\hbox{if }\beta \geq 1.
\end{cases}
$$
where (see Example \ref{exmp:Boundary_cylinder_parabolic})
\begin{align*}
\mydirac_0\!\sO_T &= \{T\}\times\bar\sO \cup (0,T)\times\partial_0\sO,
\\
\mydirac_1\!\sO_T &= \{T\}\times\bar\sO \cup (0,T)\times\partial_1\sO,
\\
\mydirac\!\sO_T &= \{T\}\times\bar\sO \cup (0,T)\times\partial\sO.
\end{align*}
Therefore, we see that the first boundary value problem of Fichera differs from the formulations in \cite{DaskalHamilton1998, Daskalopoulos_Rhee_2003, Feehan_Pop_mimickingdegen_pde} when $0<\beta<1$, where a Dirichlet boundary condition along $\mydirac_0\!\sO_T$ is replaced by the requirement that $u$ have the regularity property, $C^{2+\alpha}_s(\underline\sO_T)\cap C(\bar\sO_T)$, up to the boundary portion $\mydirac_0\!\sO_T$. This boundary regularity paradigm yields a uniqueness result which is more powerful than that suggested by the Fichera maximum principle \cite[Theorem 1.1.2]{Radkevich_2009a} for $u\in C^2(\sO_T)\cap C(\bar\sO_T)$.

We also observe that
$$
\Sigma_0\cup\Sigma_1
=
\begin{cases}
\{0\}\times\sO, &\hbox{if }0<\beta<1,
\\
\{0\}\times\sO \cup (0,T)\times\partial_0\sO, &\hbox{if }\beta \geq 1.
\end{cases}
$$
In the case of $C^2(\sO_T)$ functions on bounded open subsets $\sO_T$, we note that the Fichera maximum principle for $C^2(\sO_T)$ functions \cite[Theorem 1.1.2]{Radkevich_2009a} requires that $u\in C^2(\sO_T\cup \Sigma_0\cup\Sigma_1)\cap C(\bar\sO_T)$ and $Lu = f$ on $\sO_T\cup \Sigma_0\cup\Sigma_1$, which is \emph{stronger} than the hypothesis of our Theorem \ref{thm:Weak_maximum_principle_C2s_bounded_domain} when $\beta\geq 1$, and yields, for $r>0$,
$$
\|u\|_{C(\bar\sO_T)} \leq \frac{1}{r}\|f\|_{C(\bar\sO_T)}\vee \|g\|_{C(\bar\Sigma_2\cup\bar\Sigma_3)},
$$
where\footnote{There is a typographical error in the statement of \cite[Theorem 1.1.2]{Radkevich_2009a}, where $\Sigma_2'\cap\Sigma_3$ should be replaced by $\Sigma_2'\cup\Sigma_3$; compare \cite[Theorem 1.1.2]{Oleinik_Radkevic}}
$\Sigma_2\cup\Sigma_3$ are as given above for $\beta\geq 1$ and $0<\beta<1$.

We see that the uniqueness result, when $f=0$ on $\sO_T\cup \Sigma_0\cup\Sigma_1$, afforded by the Fichera maximum principle \cite[Theorem 1.1.2]{Radkevich_2009a} is \emph{weaker} than that of our Theorem \ref{thm:Weak_maximum_principle_C2s_bounded_domain} when $0<\beta < 1$, since we only require $g=0$ on $\mydirac_1\!\sO_T$, and \emph{not} $g=0$ on $\mydirac\!\sO_T$ to ensure that $u=0$ on $\sO_T$. Indeed, the prescription of a Dirichlet boundary condition along $\mydirac_0\!\sO_T$, when $0<\beta<1$, ensures that solutions to the first boundary value problem of Fichera are at most continuous up to $\mydirac_0\!\sO_T$ and not smooth as in \cite{DaskalHamilton1998, Daskalopoulos_Rhee_2003, Feehan_perturbationlocalmaxima, Feehan_Pop_mimickingdegen}.

%
%

\bibliography{mfpde}

\def\cprime{$'$} \def\polhk#1{\setbox0=\hbox{#1}{\ooalign{\hidewidth
  \lower1.5ex\hbox{`}\hidewidth\crcr\unhbox0}}} \def\cprime{$'$}
  \def\cprime{$'$} \def\cprime{$'$}
  \def\lfhook#1{\setbox0=\hbox{#1}{\ooalign{\hidewidth
  \lower1.5ex\hbox{'}\hidewidth\crcr\unhbox0}}} \def\cprime{$'$}
  \def\cprime{$'$} \def\cprime{$'$} \def\cprime{$'$} \def\cprime{$'$}
\providecommand{\bysame}{\leavevmode\hbox to3em{\hrulefill}\thinspace}
\providecommand{\MR}{\relax\ifhmode\unskip\space\fi MR }
\providecommand{\MRhref}[2]{%
  \href{http://www.ams.org/mathscinet-getitem?mr=#1}{#2}
}
\providecommand{\href}[2]{#2}
\begin{thebibliography}{10}

\bibitem{Adams_1975}
R.~A. Adams, \emph{Sobolev spaces}, Academic Press, Orlando, FL, 1975.

\bibitem{Bensoussan_Lions}
A.~Bensoussan and J.~L. Lions, \emph{Applications of variational inequalities
  in stochastic control}, North-Holland, New York, 1982.

\bibitem{Ciomaga_2012}
A.~Ciomaga, \emph{On the strong maximum principle for second-order nonlinear
  parabolic integro-differential equations}, Adv. Differential Equations
  \textbf{17} (2012), no.~7-8, 635--671, arXiv:1006.2607.

\bibitem{Crandall_Ishii_Lions_1992}
M.~G. Crandall, H.~Ishii, and P-L. Lions, \emph{User's guide to viscosity
  solutions of second order partial differential equations}, Bull. Amer. Math.
  Soc. (N.S.) \textbf{27} (1992), 1--67.

\bibitem{Daskalopoulos_Feehan_optimalregstatheston}
P.~Daskalopoulos and P.~M.~N. Feehan, \emph{${C}^{1,1}$ regularity for
  degenerate elliptic obstacle problems in mathematical finance},
  arXiv:1206.0831.

\bibitem{Daskalopoulos_Feehan_statvarineqheston}
\bysame, \emph{Existence, uniqueness, and global regularity for variational
  inequalities and obstacle problems for degenerate elliptic partial
  differential operators in mathematical finance}, arXiv:1109.1075.

\bibitem{DaskalHamilton1998}
P.~Daskalopoulos and R.~Hamilton, \emph{{$C^\infty$}-regularity of the free
  boundary for the porous medium equation}, J. Amer. Math. Soc. \textbf{11}
  (1998), 899--965.

\bibitem{Daskalopoulos_Rhee_2003}
P.~Daskalopoulos and E.~Rhee, \emph{Free-boundary regularity for generalized
  porous medium equations}, Commun. Pure Appl. Anal. \textbf{2} (2003),
  481--494.

\bibitem{DuffyFDM}
D.~J. Duffy, \emph{Finite difference methods in financial engineering}, Wiley,
  New York, 2006.

\bibitem{Ekstrom_Tysk_bcsftse}
E.~Ekstr{\"o}m and J.~Tysk, \emph{Boundary conditions for the single-factor
  term structure equation}, Ann. Appl. Probab. \textbf{21} (2011), 332--350.

\bibitem{Epstein_Mazzeo_annmathstudies}
C.~L. Epstein and R.~Mazzeo, \emph{Degenerate diffusion operators arising in
  population biology}, Annals of Mathematics Studies, Princeton University
  Press, Princeton, NJ, 2013, arXiv:1110.0032.

\bibitem{Evans}
L.~C. Evans, \emph{Partial differential equations}, American Mathematical
  Society, Providence, RI, 1998.

\bibitem{Feehan_classical_perron_parabolic}
P.~M.~N. Feehan, \emph{A classical {P}erron method for existence of smooth
  solutions to boundary value and obstacle problems for degenerate-parabolic
  operators via holomorphic maps}, in preparation.

\bibitem{Feehan_maximumprinciple_v1}
\bysame, \emph{Partial differential operators with non-negative characteristic
  form, maximum principles, and uniqueness for boundary value and obstacle
  problems}, Communications in Partial Differential Equations, to appear,
  arXiv:1204.6613v1.

\bibitem{Feehan_perturbationlocalmaxima}
\bysame, \emph{Perturbations of local maxima and comparison principles for
  boundary-degenerate linear differential equations}, arXiv:1305.5098.

\bibitem{Feehan_Pop_regularityweaksoln}
P.~M.~N. Feehan and C.~A. Pop, \emph{Degenerate elliptic operators in
  mathematical finance and {H\"o}lder continuity for solutions to variational
  equations and inequalities}, arXiv:1110.5594.

\bibitem{Feehan_Pop_mimickingdegen}
\bysame, \emph{Degenerate-parabolic partial differential equations with
  unbounded coefficients, martingale problems, and a mimicking theorem for
  {I}t\^o processes}, arXiv:1112.4824v1.

\bibitem{Feehan_Pop_mimickingdegen_pde}
\bysame, \emph{A {S}chauder approach to degenerate-parabolic partial
  differential equations with unbounded coefficients}, Journal of Differential
  Equations \textbf{254} (2013), 4401--4445, arXiv:1112.4824.

\bibitem{Fichera_1956}
G.~Fichera, \emph{Sulle equazioni differenziali lineari ellittico-paraboliche
  del secondo ordine}, Atti Accad. Naz. Lincei. Mem. Cl. Sci. Fis. Mat. Nat.
  Sez. I. (8) \textbf{5} (1956), 1--30.

\bibitem{Fichera_1960}
\bysame, \emph{On a unified theory of boundary value problems for
  elliptic-parabolic equations of second order}, Boundary problems in
  differential equations, Univ. of Wisconsin Press, Madison, 1960, pp.~97--120.

\bibitem{Friedman_1958}
A.~Friedman, \emph{Remarks on the maximum principle for parabolic equations and
  its applications}, Pacific J. Math. \textbf{8} (1958), 201--211. \MR{0102655
  (21 \#1444)}

\bibitem{FriedmanPDE}
\bysame, \emph{Partial differential equations of parabolic type}, Prentice
  Hall, New York, 1964.

\bibitem{GilbargTrudinger}
D.~Gilbarg and N.~Trudinger, \emph{Elliptic partial differential equations of
  second order}, second ed., Springer, New York, 1983.

\bibitem{Heston1993}
S.~Heston, \emph{A closed-form solution for options with stochastic volatility
  with applications to bond and currency options}, Review of Financial Studies
  \textbf{6} (1993), 327--343.

\bibitem{Hill_1971}
C.~D. Hill, \emph{A sharp maximum principle for degenerate elliptic-parabolic
  equations.}, Indiana Univ. Math. J. \textbf{20} (1970/1971), 213--229.

\bibitem{Koch}
H.~Koch, \emph{Non-{E}uclidean singular integrals and the porous medium
  equation}, Habilitation Thesis, University of Heidelberg, 1999,
  \url{www.mathematik.uni-dortmund.de/lsi/koch/publications.html}.

\bibitem{Krylov_LecturesHolder}
N.~V. Krylov, \emph{Lectures on elliptic and parabolic equations in {H}\"older
  spaces}, American Mathematical Society, Providence, RI, 1996.

\bibitem{Krylov_LecturesSobolev}
\bysame, \emph{Lectures on elliptic and parabolic equations in {S}obolev
  spaces}, American Mathematical Society, Providence, RI, 2008.

\bibitem{Kusano_1963}
T.~Kusano, \emph{Remarks on some properties of solutions of some boundary value
  problems for quasi-linear parabolic and elliptic equations of the second
  order}, Proc. Japan Acad. \textbf{39} (1963), 217--222. \MR{0164150 (29
  \#1449)}

\bibitem{LadyzenskajaSolonnikovUralceva}
O.~A. Ladyzenskaja, V.~A. Solonnikov, and N.~N. Ural'ceva, \emph{Linear and
  quasi-linear equations of parabolic type}, American Mathematical Society,
  Providence, RI, 1995.

\bibitem{Lieberman}
G.~M. Lieberman, \emph{Second order parabolic differential equations}, World
  Scientific Publishing Co. Inc., River Edge, NJ, 1996.

\bibitem{Nazarov_2012}
A.~I. Nazarov, \emph{A centennial of the {Z}aremba-{H}opf-{O}leinik lemma},
  SIAM J. Math. Anal. \textbf{44} (2012), no.~1, 437--453. \MR{2888295}

\bibitem{Nirenberg_1953}
L.~Nirenberg, \emph{A strong maximum principle for parabolic equations}, Comm.
  Pure. Applied. Math. \textbf{6} (1953).

\bibitem{Oleinik_Radkevic}
O.~A. Ole{\u\i}nik and E.~V. Radkevi{\v{c}}, \emph{Second order equations with
  nonnegative characteristic form}, Plenum Press, New York, 1973.

\bibitem{Pozio_Punzo_Tesei_2008}
M.~A. Pozio, F.~Punzo, and A.~Tesei, \emph{Criteria for well-posedness of
  degenerate elliptic and parabolic problems}, J. Math. Pures Appl. (9)
  \textbf{90} (2008), 353--386.

\bibitem{Punzo_Tesei_2009a}
F.~Punzo and A.~Tesei, \emph{On the refined maximum principle for degenerate
  elliptic and parabolic problems}, Nonlinear Anal. \textbf{70} (2009),
  3047--3055.

\bibitem{Punzo_Tesei_2009b}
\bysame, \emph{Uniqueness of solutions to degenerate elliptic problems with
  unbounded coefficients}, Ann. Inst. H. Poincar\'e Anal. Non Lin\'eaire
  \textbf{26} (2009), 2001--2024.

\bibitem{Radkevich_2009a}
E.~V. Radkevi{\v{c}}, \emph{Equations with nonnegative characteristic form.
  {I}}, J. Math. Sci. \textbf{158} (2009), 297--452.

\bibitem{Radkevich_2009b}
\bysame, \emph{Equations with nonnegative characteristic form. {II}}, J. Math.
  Sci. \textbf{158} (2009), 453--604.

\bibitem{Rodrigues_1987}
J-F. Rodrigues, \emph{Obstacle problems in mathematical physics},
  North-Holland, New York, 1987.

\bibitem{Taira_2004}
K.~Taira, \emph{Semigroups, boundary value problems and {M}arkov processes},
  Springer, Berlin, 2004.

\bibitem{Troianiello}
G.~M. Troianiello, \emph{Elliptic differential equations and obstacle
  problems}, Plenum Press, New York, 1987.

\bibitem{Trudinger_1977}
N.~S. Trudinger, \emph{Maximum principles for linear, non-uniformly elliptic
  operators with measurable coefficients}, Math. Z. \textbf{156} (1977),
  291--301.

\bibitem{Wang_1990}
X.~Wang, \emph{A remark on strong maximum principle for parabolic and elliptic
  systems}, Proc. Amer. Math. Soc. \textbf{109} (1990), no.~2, 343--348.
  \MR{1019284 (90k:35048)}

\bibitem{Wu_Yin_Wang_2006}
Z.~Wu, J.~Yin, and C.~Wang, \emph{Elliptic \& parabolic equations}, World
  Scientific, Hackensack, NJ, 2006. \MR{2309679 (2007m:35003)}

\bibitem{ZvanForsythVetzal}
R.~Zvan, P.~A. Forsyth, and K.~R. Vetzal, \emph{Penalty method for {A}merican
  options with stochastic volatility}, Journal of Computational and Applied
  Mathematics \textbf{91} (1998), 199--218.

\end{thebibliography}
\bibliographystyle{amsplain}

\end{document}